\documentclass[11pt]{amsart}

\usepackage[T1]{fontenc}
\usepackage{amsmath,amsthm,amssymb,mathtools}
\usepackage{palatino}
\usepackage{hyperref}
\usepackage{bm}
\usepackage{xcolor}

\usepackage{parskip}
\usepackage{cite}
\usepackage{comment}
\usepackage{setspace}
\usepackage{lipsum}
\usepackage{enumitem} 

\DeclareMathAlphabet{\mathcal}{OMS}{cmsy}{m}{n}

\hypersetup{
  colorlinks=true,
  linktoc=all,
  linkcolor=blue,
  citecolor=red, 
  urlcolor=blue,
  pdftitle={Pinned configuration sets and slicing theorems via Fourier integral operators},
  pdfsubject={Pinned configurations, slicing, Fourier integral operators, restricted projections}
}

\makeatletter
\renewcommand\subsection{\@startsection{subsection}{2}{0pt}%
  {1ex}{0.5ex}{\normalfont\normalsize\bfseries}}
\makeatother
\numberwithin{equation}{section}
\allowdisplaybreaks
\hypersetup{unicode, psdextra}
\newcommand{\noopsort}[1]{}

\theoremstyle{plain}
\newtheorem{theorem}{Theorem}[section]
\newtheorem{proposition}[theorem]{Proposition}
\newtheorem{lemma}[theorem]{Lemma}
\newtheorem{corollary}[theorem]{Corollary}
\newtheorem*{conjecture*}{Conjecture}

\theoremstyle{definition}
\newtheorem{definition}[theorem]{Definition}
\newtheorem{example}[theorem]{Example}
\theoremstyle{remark}
\newtheorem{remark}[theorem]{Remark}
\newtheorem*{remark*}{Remark}

\newcommand{\R}{\mathbb{R}}

\newcommand{\e}{\varepsilon}
\newcommand{\bt}{\mathbf{t}}

\newcommand{\calR}{\mathcal{R}}
\newcommand{\calL}{\mathcal{L}}

\newcommand{\calS}{\mathcal{S}}

\newcommand{\bzero}{\mathbf{0}}
\DeclareMathOperator{\Int}{Int}
\DeclareMathOperator{\supp}{supp}
\DeclareMathOperator{\rank}{rank}
\DeclareMathOperator{\Ker}{Ker}
\DeclareMathOperator{\corank}{corank}
\newcommand{\Bphi}{\mathbf B_{\Phi}}
\newcommand{\Cphi}{\mathbf C_{\Phi}}
\newcommand{\Tphi}{\mathbf T_{\Phi}}
\newcommand{\calG}{\mathcal{G}}

\title{Pinned configuration sets and slicing theorems via Fourier integral operators}
\author{A. Greenleaf}
\address{Department of Mathematics, University of Rochester, NY 14627, USA}
\email{allan@math.rochester.edu}
 \author{E. A. Palsson}
\address{Department of Mathematics, Virginia Tech, Blacksburg, VA 24061, USA} \email{palsson@vt.edu}
 \author{M-Q. Pham}
\address{Department of Mathematics, Virginia Tech, Blacksburg, VA 24061, USA} 
\email{mqpham@vt.edu}
\author{G. Psaromiligkos}
\address{Department of Mathematics, University of Thessaly, Lamia, Greece}
\email{gpsaromiligkos@uth.gr}
\date{\today}

\begin{document}
\thispagestyle{empty}

\begin{abstract} 
Let $X\subset\R^{d_X}$ and $Y\subset\R^{d_Y}$ be open, where $d_X,d_Y\geq2$ and $1\leq p\leq\min\{d_X,d_Y\}$. Let $\Phi:X\times Y\to\mathbb R^p$ be a smooth mapping, and suppose that $D_x\Phi$ and $D_y\Phi$ have rank $p$.
For compact sets $A\subset X$ and $B\subset Y$, we give lower bounds on $\dim_H A+\dim_H B$ ensuring that the pinned configuration set
\[\Delta_\Phi^y(A)=\{\Phi(x,y):x\in A\}\]
has positive Lebesgue measure or nonempty interior for some $y\in B$. 
The dimensional thresholds are expressed in terms of uniform $L^2$-Sobolev estimates for the generalized Radon transforms associated with $\Phi$. Our results extend the Fourier integral operator approach developed by Greenleaf, Iosevich, and Taylor \cite{GreenleafIosevichTaylor2021} for unpinned $2$-point configuration problems to the pinned setting.

We also prove a nonlinear projection theorem that remains valid even when the Peres-Schlag transversality condition fails. Under the smoothing hypotheses used for the pinned configuration theorem, we obtain lower bounds for the Hausdorff dimensions of the corresponding slices of $A$ by the level sets of $\Phi$. 
Applications include distances from points to hyperplanes,
spheres and lines; multiparameter distance sets; and a restricted
projection family introduced by Harris.
\end{abstract}

\maketitle

\section{Introduction}
\subsection{Pinned $2$-point configuration sets}
Let $d_X,d_Y\geq  2$ and $1\leq  p\leq \min\{d_X,d_Y\}$. Let $X\subset \R^{d_X}$, $Y\subset \R^{d_Y}$ be open sets. Consider a smooth map $\Phi:X\times Y\to\R^p$ such that $D_x\Phi$ and $D_y\Phi$ both have rank $p$ everywhere. Given a compact set $A\subset X$ and a point $y\in Y$, define the pinned $\Phi$-\textit{configuration set} by
\begin{align*}
    \Delta_{\Phi}^y(A):=\{ \Phi(x,y): x\in A\}\subset \R^p.
\end{align*}
Given a compact set $B\subset Y$, we find dimensional conditions on $A$ and $B$ ensuring that $\Delta_{\Phi}^y(A)$ has positive Lebesgue measure, or even has nonempty interior, for some $y\in B$.

This question is a variant of the Falconer distance problem, a classical problem in geometric measure theory and harmonic analysis.
In 1985, Falconer \cite{Falconer85} introduced the distance set problem, showing that if a compact set $A\subset\R^d$, $d\geq  2$, has Hausdorff dimension $\dim_H A>\frac{d+1}{2}$, then its
\textit{distance set}
\begin{align*}
    \Delta(A):=\left\{\, \vert x-y\vert: x,y\in A\right\}
\end{align*}
has positive Lebesgue measure. In addition, he conjectured that the dimensional threshold should be $\frac{d}{2}$, for $d\geq  2$. The exponent $\frac{d}{2}$ is sharp as a possible general threshold. Over the past several decades, a variety of analytic and geometric techniques have been developed to lower the dimensional threshold that is sufficient for the distance set $\Delta(A)$ to have positive Lebesgue measure. 
Previously established sufficient thresholds include $\dim_HA>5/4$
when $d=2$, proved by Guth, Iosevich, Ou, and Wang
\cite{Guth_Iosevich_Ou_Wang_2019}, and
$\dim_HA>d/2+1/4-1/(8d+4)$ when $d\geq3$, proved by Du, Ou,
Ren, and Zhang \cite{Duetal23a,Duetal23b}. Very recently, OpenAI posted a preprint claiming a proof of the Falconer distance conjecture
in all dimensions \cite{OpAI26}.

More generally, one can ask Falconer type questions for other geometric configurations or smooth functions. For a smooth map $\Phi:X\times Y\to\R^p$, and for compact subsets $A\subset X$ and $B\subset Y$, we can ask under what nondegeneracy assumptions on $\Phi$ there is a threshold $s_0=s_0(\Phi)$ such that if $\dim_HA+\dim_HB>s_0$, then the $\Phi$-configuration set
\begin{align*}
     \Delta_\Phi(A,B):=\{ \Phi(x,y): x\in A, y\in B\}\subset \R^p
\end{align*}
has positive Lebesgue measure. These questions and their generalizations to $k$-point configurations have been studied extensively, e.g., in \cite{ESWARATHASAN2011,Greenleaf_Iosevich_2012,Grafakos_Greenleaf_Iosevich_Palsson_15,Greenleaf_Iosevich_Liu_Palsson_2015,Greenleaf_Iosevich_Mourgoglou_2015,Bennett_Iosevich_Taylor_2016,Iosevich_Taylor_2019,Iosevich_Liu_pinned_distance_2019,Palsson_Acosta_2023,Palsson_Acosta_2025,Iosevich_Pham_Pham_Shen_2025,Shmerkin_Yavicoli_2025}, with more recent contributions including \cite{Borges_etal_2026a,Borges_etal_2026_b,Borges_etal_2026_c,Pham_Falconerfunction_2026,Pham_FIO_2026,Pham_Shen_Xue_2026}; see also the references therein. 

Another variant of the Falconer distance problem is the Mattila-Sj{\"o}lin type problem, which asks for a threshold $r_0(\Phi)>0$ such that $\dim_HA+\dim_HB>r_0$ implies $\Delta_\Phi(A,B)$ contains a nonempty open set, that is, has nonempty interior. These types of questions are motivated by the work of Mattila and Sj{\"o}lin \cite{MattilaSjolin99} for distance sets. In a series of three papers, Greenleaf, Iosevich, and Taylor \cite{GreenleafIosevichTaylor2021,GreenleafIosevichTaylor22,GreenleafIosevichTaylor2024} developed a framework to study Mattila--Sj{\"o}lin type problems for a general class of $k$-point configurations using Fourier integral operator (FIO) techniques. More precisely Greenleaf, Iosevich, and Taylor \cite{GreenleafIosevichTaylor2021} proved an unpinned nonempty interior theorem for $2$-point configuration sets, where the threshold $r_0=r_0(\Phi)$ can be expressed in terms of $d_X,d_Y, p$ and $\alpha_\Phi$, the amount of smoothing on $L^2$-based Sobolev spaces satisfied by the family of generalized Radon transforms $\mathcal{R}_{\bt}$ defined by $\Phi$. For recent Mattila--Sj{\"o}lin type results for other geometric configurations, see, e.g. \cite{ Koh_Pham_Shen_2022,Palsson_Acosta_2023,GreenleafIosevichTaylor2024,Palsson_Acosta_2025,Borges_etal_nonempty_interior_2026,Palsson_Psaromiligkos_2026}.

In this paper, we apply the FIO framework of \cite{GreenleafIosevichTaylor2021} to study the pinned variants of the Falconer type and Mattila--Sj{\"o}lin type problems for a broad class of smooth configuration maps $\Phi: X\times Y\to \R^p$. The associated family of generalized Radon transforms is defined as follows. Recall that we assume 
\begin{align}\label{eq: rank_condition}
    \rank D_x\Phi=p,\quad \rank D_y\Phi=p.
\end{align}
Thus, for each $\bt\in\R^p$, the incidence set
\[
Z_\bt=\{(x,y)\in X\times Y:\Phi(x,y)=\bt\}
\]
is a smooth submanifold of codimension $p$ (when nonempty).
For the compact sets $A\subset X$ and $B\subset Y$ under
consideration, choose $\chi\in C_0^\infty(X\times Y)$,
$\chi\geq0$, equal to one in a neighborhood of $A\times B$.
Choose a bounded open set $W\subset\R^p$ such that
$\Phi(\supp\chi)\subset W$. We consider the family of
incidence relations $Z_\bt$ for $\bt\in W$.
Define the generalized Radon transform
\begin{align}\label{eq: def_generalized_Radon_trans}
   \mathcal R_{\bt}f(y)= \int_{Z_\bt^y}
f(x)\chi(x,y)\,d\sigma_{y,\bt}(x),
\end{align}
where for each $y\in Y$, $Z_\bt^y:=\{x: (x,y)\in Z_\bt\}\subset X$ is a smooth surface of codimension $p$, and $d\sigma_{y,\bt}$ is the Leray measure on $Z_\bt^y$.

It is known that each $\calR_\bt$ is a Fourier integral operator (FIO), of order $m=\frac{p}{2}-\frac{d_X+d_Y}{4}$, and associated to a canonical relation $\Lambda_\bt$, where $\Lambda_\bt=(N^\ast Z_\bt)^{\prime}\subset T^\ast \R^{d_X}\times T^\ast\R^{d_Y}$, the twisted conormal bundle of $Z_\bt$ (see Section \ref{sec: background} for more details). For our result, we assume the regularity condition on the $\calR_\bt$, that is, there is an absolute smoothing $\alpha_\Phi$, determined by the geometry of the canonical relations $\Lambda_\bt$, 
\[ 0\leq  \alpha_\Phi\leq  \frac{\min\{d_X,d_Y\}-p}{2},\] 
such that, for any $s\in \R$, one has
\begin{align}\label{eq: Rt_smoothing_condition_intro}
    \Vert \calR_{\bt}\Vert_{L_s^2(X)\to L_{s+\alpha_\Phi}^2(Y)}\leq  C_s,
\end{align}
with $C_s$ bounded uniformly as $\bt$ varies over $W$.

Our main result is the following.
\begin{theorem}\label{thm: main_pinned_configuration}
    Let $d_X,d_Y\geq  2$, and $1\leq  p\leq  \min\{d_X,d_Y\}$. Let $X\subset \R^{d_X}$, $Y\subset \R^{d_Y}$ be open sets. Suppose that $\Phi: X\times Y\to \R^p$ is a smooth map that satisfies \eqref{eq: rank_condition} and that the generalized Radon transforms $\calR_\bt$ in \eqref{eq: def_generalized_Radon_trans} satisfy \eqref{eq: Rt_smoothing_condition_intro} with smoothing order $\alpha_\Phi$. 

    Assume that $A\subset X$, $B\subset Y$ are compact sets with $\dim_HA,\dim_HB>0$. Then the following statements hold:
    \begin{itemize}
        \item[(i)] If $\alpha_\Phi\in (0,(\min\{d_{X},d_Y\}-p)/2]$ and $\dim_HA+\dim_HB>d_X+d_Y-2\alpha_\Phi$, then there exists $y\in B$ such that $\calL^p(\Delta_\Phi^y(A))>0$.
        \item[(ii)] If $\alpha_\Phi\in (p/2,(\min\{d_{X},d_Y\}-p)/2]$ and $\dim_HA+\dim_HB>d_X+d_Y-2\alpha_\Phi+p$, then there exists $y\in B$ such that $\Delta_\Phi^y(A)$ has nonempty interior. 
    \end{itemize}
\end{theorem}
\begin{remark}
We remark that the roles of $X$ and $Y$ can be interchanged. Indeed, the
generalized Radon transforms associated with
$\widetilde{\Phi}(y,x)=\Phi(x,y)$ are the adjoints of
$\calR_\bt$ and have the same $L^2$-Sobolev smoothing
order by duality. Thus, Theorem~\ref{thm: main_pinned_configuration}
also applies to
\[
\Delta_\Phi^x(B):=\{\Phi(x,y):y\in B\},
\]
for $x\in A$, with the roles of $A$ and $B$
interchanged in the hypotheses. For example, in the point--hyperplane distance setting of Corollary~\ref{cor:pinned-point-hyperplane},
either the point or the hyperplane may be pinned.
\end{remark}
\begin{remark}
    For scalar configuration maps, Palsson and Psaromiligkos \cite[Proposition 3.1 and Remark 3.3]{Palsson_Psaromiligkos_2026} obtained pinned positive measure and nonempty interior conclusions under smoothing assumptions for generalized Radon transforms. The present paper 
 treats vector-valued configuration maps and establishes an accompanying quantitative slicing theorem, which we will discuss in the next section.
\end{remark}

\begin{remark}
Theorem \ref{thm: main_pinned_configuration} gives a pinned counterpart of the $2$-point configuration set theorem of Greenleaf, Iosevich, and Taylor \cite[Theorem 1.6]{GreenleafIosevichTaylor2021}. To make the distinction precise, write $T_\Phi=d_X+d_Y-2\alpha_\Phi$.
Their theorem gives nonempty interior for the unpinned set $\Delta_\Phi(A,B)$ when $\dim_HA+\dim_HB>T_\Phi$. Theorem \ref{thm: main_pinned_configuration} gives positive Lebesgue measure for the pinned set $\Delta_\Phi^y(A)$, with some $y\in B$, under the same dimensional condition, and nonempty interior for such a pinned set when $\dim_HA+\dim_HB>T_\Phi+p$.
Note that the unpinned interior result on its own does not imply either conclusion for the pinned problem.
\end{remark}
When $\Phi: \R^d\times \R^d\to \R$ is the distance function, the pinned distance set problem was first studied by Peres and Schlag \cite{Peres_Schlag_2000}, who considered a very large class of functions $\Phi(x,y)$ satisfying transversality conditions that generalize orthogonal projections $P_e(x)=x\cdot e$, $e\in S^{d-1}$. The Euclidean distance function belongs to this class, and Peres and Schlag \cite{Peres_Schlag_2000} obtained a sufficient threshold $\dim_HA>\frac{d}{2}+\frac{1}{2}$ guaranteeing that there exists a point $y\in A$ for which the pinned distance set at $y$,
\begin{align*}
    \Delta^y(A):=\{|x-y|: x\in A\},
\end{align*}
has positive Lebesgue measure. Some pinned distance estimates also control exceptional pins relative
to Frostman measures on $A$; see, for example,
\cite{Iosevich_Liu_pinned_distance_2019}.
Furthermore, in this pinned setting, Peres and Schlag \cite{Peres_Schlag_2000} also showed that if $\dim_HA>\frac{d}{2}+1$, then there exists a point $y\in A$ such that the pinned distance set $\Delta^y(A)$ has nonempty interior. However, when $d=2$, the condition in \cite{Peres_Schlag_2000} is vacuous and provides no information for the nonempty interior problem. The first nontrivial threshold for nonempty interior of the pinned distance set in the plane was obtained by Borges, Foster, Ou, and Palsson in \cite{Borges_etal_nonempty_interior_2026}, where the authors showed that if $A\subset \R^2$ with $\dim_HA>\frac{7}{4}$, then there exists $y\in A$ such that $\Delta^y(A)$ has nonempty interior. They also obtained the improved threshold $12/5$ when $d=3$. For $d\geq4$, this was improved by Borges, Ou, and Pasquariello \cite{Borges_etal_2026_c} who obtained the threshold $\frac{d}{2}+\frac{3}{4}+\frac{3}{4(2d-1)}$. The framework in \cite{Borges_etal_nonempty_interior_2026} relies on $L^p\to L^p$ local smoothing inequalities, originating in the work of Iosevich and Liu \cite{Iosevich_Liu_pinned_distance_2019}. Iosevich and Liu developed local smoothing methods to study exceptional pins, recovered the results of Peres and Schlag, and extended them to a variety of metrics satisfying certain curvature conditions. 

\subsection{Nonlinear projections and rank conditions}

For $y\in Y$, write $\Phi_y(x)=\Phi(x,y)$, so that
$\Phi_y(A)=\Delta_\Phi^y(A)$.
Peres and Schlag \cite{Peres_Schlag_2000} obtained projection
theorems under transversality assumptions on the family
$\{\Phi_y\}_{y\in Y}$. Such assumptions may fail even when
the associated generalized Radon transforms satisfy Sobolev
smoothing estimates. We obtain projection estimates under
a rank condition on the augmented mixed Hessian defined below.

For $(x,y)\in X\times Y$ and $\tau\in\R^p\setminus\bzero$, define the
augmented mixed Hessian
\[
\mathbf B_\Phi(x,y,\tau)=
\begin{bmatrix}
0&D_x\Phi(x,y)\\
D_y\Phi(x,y)^T&D_yD_x(\tau\cdot\Phi)(x,y)
\end{bmatrix}.
\]
\begin{samepage}
\begin{theorem}
Let $A\subset X$ be compact, and write $ s=\dim_HA>0$. Put $d_{\min}=\min\{d_X,d_Y\}$.
Suppose that, for some integer $ 2p\leq  r\leq  p+d_{\min}$, one has
\begin{equation*}
 \rank(\Bphi(x,y,\tau))\geq  r
\end{equation*}
for every $(x,y)\in X\times Y$ and every
$\tau\in\R^p\setminus\bzero$. Then the following hold.

\begin{enumerate}
\item If $0<s<p$, then, for every $0<u<s$,
\begin{equation*}
 \dim_H\bigl\{
 y\in Y:\dim_H\Phi_y(A)\leq  u
 \bigr\}
 \leq  d_X+d_Y+p-r-s+u.
\end{equation*}

\item If $s>p$, then
\begin{equation*}
 \dim_H\bigl\{
 y\in Y:\calL^p(\Phi_y(A))=0
 \bigr\}
 \leq  d_X+d_Y+2p-r-s.
\end{equation*}

\item If $s>2p$, then
\begin{equation*}
 \dim_H\bigl\{
 y\in Y:\Int(\Phi_y(A))=\emptyset
 \bigr\}
 \leq  d_X+d_Y+3p-r-s.
\end{equation*}

\item Let $B\subset Y$ be compact and put
\[
 \kappa_r
 =\dim_HA+\dim_HB-d_X-d_Y-2p+r.
\]
If $\kappa_r>0$, then, for every
$0<\e<\kappa_r$, there exists a Frostman probability
measure $\lambda_\e$ supported on $B$ such that, for
$\lambda_\e$-a.e.\ $y$,
\[
 \calL^p\left(
 \left\{\bt\in\R^p:
 \dim_H\bigl(
 A\cap\Phi(\cdot,y)^{-1}(\bt)
 \bigr)\geq \kappa_r-\e
 \right\}\right)>0.
\]
\end{enumerate}
\end{theorem}
\end{samepage}
We also give a criterion for failure of Peres--Schlag transversality
of order zero, in terms of collisions of distinct input points
and limits of secant directions; see
Theorem~\ref{thm:PS-dichotomy}. Thus the rank-based conclusions
can apply to families for which this transversality condition fails.

\subsection{Dimensions of slices}
Let $\{\pi_\omega:\R^d\to\R\}_{\omega\in\Omega}$ be a family
of mappings. For each fixed $\omega$, the sets
$A\cap\pi_\omega^{-1}(t)$, $t\in\R$, are the slices of $A$
determined by $\pi_\omega$. In addition to estimating the size
of the image $\pi_\omega(A)$, one can ask how large these slices
are. A classical result is due to Marstrand.
\begin{samepage}
\begin{theorem}[Marstrand \cite{Marstrand54}]
Suppose that $A\subset \R^2$ is a Borel set, and let $\pi_e(x)=x\cdot e$, for $e\in S^{1}$. Then for almost all $e\in S^1$,
\begin{align}
    \dim_H(\pi_e(A))=\dim_HA&\quad \text{ if }\dim_HA\leq  1,\notag\\
    \calL^1(\pi_e(A))>0& \quad \text{ if }\dim_HA>1.\label{eq: Marstrand_proj}
\end{align}
Moreover, if $s=\dim_H A>1$ and $\mathcal H^s(A)>0$,
then, for a.e.  $e\in S^1$,
\begin{align}\label{eq: Marstrand_slice}
    \calL^1\big(\{t\in \R: \dim_H(\pi_e^{-1}(t)\cap A)=\dim_HA-1\}\big)>0.
\end{align}
\end{theorem}
\end{samepage}
    Various generalizations of this result have been obtained over the years. See, for example, \cite{Mattila2015}, and the references therein. In \cite{Peres_Schlag_2000}, Peres and Schlag generalized \eqref{eq: Marstrand_proj} from orthogonal projections to a large class of mappings called generalized projections, satisfying transversality conditions. This class of mappings includes, in particular, Euclidean distances $\pi_y(x)=|x-y|$, $x,y\in \R^d$. Slicing theorems for generalized projections were proved in \cite{Orponen_2014_slicing} and \cite{Bushling_2025_slicing}.

    Using the estimates in the proof of Theorem~\ref{thm: main_pinned_configuration}, we establish a generalization of \eqref{eq: Marstrand_slice} for a large class of smooth mappings including those that fail to satisfy transversality conditions in the sense of Peres and Schlag. 
    \begin{samepage}
\begin{theorem}\label{thm: pinned_slicing}
    Assume that the hypotheses of Theorem \ref{thm: main_pinned_configuration} hold. Suppose that $A\subset X$, $B\subset Y$ are compact sets with
    \begin{align*}
        s_A=\dim_HA, \quad s_B=\dim_HB, \quad \kappa=s_A+s_B- d_X-d_Y+2\alpha_\Phi>0.
    \end{align*}
    Then for every $0<\e<\kappa$, there exists a Frostman probability measure $\lambda_\e$ supported on $B$ such that, for $\lambda_\e$-a.e.  $y\in B$,
    \begin{align*}
        \calL^p\big(\Delta_\Phi^y(A)\big)>0,
    \end{align*}
    and 
    \begin{align*}
        \calL^p\big(\{\bt\in \R^p: \dim_H\{x\in A: \Phi(x,y)=\bt\}\geq  \kappa-\e\}\big)>0.
    \end{align*}
\end{theorem}
\end{samepage}
\subsection{A family of restricted projections}
We consider a higher dimensional restricted projection family introduced by Harris \cite{Harris2022}. Let $d\geq 4$ be even, and let
\begin{align*}
A:\R^d\longrightarrow\R^d
\end{align*}
be an orthogonal skew-symmetric linear map, that is $A^*A=I$ and $A^*=-A$. Write $S^{d-1}_+=\{v\in S^{d-1}:v_d>0\}$ for the open upper hemisphere. Thus, for every $v\in S^{d-1}_+$,
\begin{align*}
v\cdot Av=0,\qquad |Av|=1.
\end{align*}
For each $v\in S^{d-1}_+$, let $\pi_v:\R^{d+1}\to W_v$ denote the orthogonal projection onto the 2-dimensional plane
\begin{align*}
W_v = \operatorname{span}\left\{\frac1{\sqrt2}(v,-1), (Av,0) \right\} \subset\R^{d+1}.
\end{align*}
Since the two displayed vectors are orthonormal, after identifying
$W_v$ with $\R^2$ we may write
\begin{align*}
\pi_{v}(x) = \left( \left\langle x,\frac1{\sqrt2}(v,-1) \right\rangle, \left\langle x,(Av,0)\right\rangle \right), \qquad x=(x',x_{d+1})\in\R^{d+1}.
\end{align*}
Harris proved that this restricted family fails to satisfy the Peres-Schlag transversality condition, and hence the Peres-Schlag projection theorem does not apply to this family. Nevertheless, he proved that if $E\subset\R^{d+1}$ is an analytic set with $\dim_HE>2$, then
\begin{align*}
\mathcal L^2(\pi_{v}(E))>0,
\end{align*}
for a.e. $v\in S^{d-1}_+$.

In Section~\ref{sec: projection_nontranversal_interior}, we apply the FIO framework developed above to obtain a nonempty interior conclusion and a quantitative bound for the exceptional set of parameters for which $\pi_{v}(E)$ has empty interior.
\begin{samepage}
\begin{theorem}
\label{thm: Harris_projection_interior}
Let $d\geq4$ be even and let $A$ be the orthogonal skew-symmetric map fixed above. Let $E\subset\R^{d+1}$ be a compact set with $\dim_HE=s>0$. Then 
\begin{align}
\label{eq: Harris_projection_interior_exceptional}
\dim_H\left\{ v\in S^{d-1}_+: \Int(\pi_{v}(E))=\emptyset \right\} \leq  \min\left\{ d-1,\, d+\frac72-s \right\}.
\end{align}
In particular, if $\dim_HE>\frac{9}{2}$, then $\Int(\pi_{v}(E))\neq\emptyset$,
for $\mathcal H^{d-1}$-a.e.  $v\in S^{d-1}_+$.
\end{theorem}
\end{samepage}
\begin{remark}
For this projection family, extending Harris's weighted Fourier restriction argument to higher Sobolev regularity might yield better dimensional thresholds for nonempty interior in some ambient dimensions. We do not pursue such a refinement here.
\end{remark}

\subsection{Examples}
We record consequences of Theorems
\ref{thm: main_pinned_configuration} and
\ref{thm: pinned_slicing} for configurations considered in
\cite[Corollaries 1.4 and 2.1-2.3]{GreenleafIosevichTaylor2021}.

\begin{corollary}[Pinned point-to-hyperplane distances]
\label{cor:pinned-point-hyperplane}
Let $d\geq 2$, let $A$ be a compact family of affine
hyperplanes in $\R^d$, and let $B\subset\R^d$ be compact.
For each $y\in \R^d$, denote
\begin{align*}
\Delta_{\mathrm{hyp}}^y(A) =\{\operatorname{dist}(y,H):H\in A\}.
\end{align*}
Then the following statements hold:
\begin{itemize}
    \item[(i)] If $\dim_HA+\dim_HB>d+1$, then for every
$0<\e<\dim_HA+\dim_HB-(d+1)$ there exists a Frostman probability measure
$\lambda_\e$ on $B$ such that, for $\lambda_\e$-a.e. $y$,
\begin{align*}
\calL^1\left( \left\{t>0: \dim_H\{H\in A:\operatorname{dist}(y,H)=t\}
\geq \dim_HA+\dim_HB-(d+1)-\e \right\}\right)>0.
\end{align*}
    \item[(ii)] If $\dim_HA+\dim_HB>d+2$, then $\Delta_{\mathrm{hyp}}^y(A)$
contains a nonempty open interval for some $y\in B$.
\end{itemize}
\end{corollary}
\begin{corollary}[Pinned point-to-sphere distances]
Let $d\geq 2$, and identify the space of spheres in $\R^d$ with
\begin{align*}
\mathcal S_d=\R^d\times(0,\infty),
\end{align*}
where $(a,r)\in \calS_d$ represents the sphere with center $a$ and radius $r$.
Let $A\subset\mathcal S_d$ and $B\subset\R^d$ be compact, and assume
that
\begin{align*} a\neq y \qquad \text{for every }(a,r)\in A,\ y\in B.
\end{align*}
For each $y\in\R^d$, denote
\begin{align*}
\Delta_{\mathrm{sph}}^y(A) = \left\{ \big||a-y|-r\big|:(a,r)\in A \right\}.
\end{align*}
Then the following hold:
\begin{enumerate}[label=\textup{(\roman*)}]
\item
If $\dim_HA+\dim_HB>d+2$, then for every $0<\e<\dim_HA+\dim_HB-(d+2)$, there exists a Frostman probability measure $\lambda_\e$ on $B$ such that, for $\lambda_\e$-a.e. $y$,
\begin{align*}
\calL^1\left(\left\{t>0:\dim_H\left\{(a,r)\in A:\big||a-y|-r\big|=t\right\}
\geq \dim_HA+\dim_HB-(d+2)-\e\right\}\right)>0.
\end{align*}

\item
If $\dim_HA+\dim_HB>d+3$, then $\Delta_{\mathrm{sph}}^y(A)$ contains a nonempty open interval
for some $y\in B$.
\end{enumerate}
\end{corollary}

\begin{corollary}[Pinned point-to-line distances]
Let $d\geq 3$, let $A$ be a compact family of affine lines in $\R^d$, and let $B\subset\R^d$ be compact. Assume that $y\notin L$ for every $L\in A$ and $y\in B$. For each $y\in\R^d$, denote
\begin{align*}
\Delta_{\mathrm{line}}^y(A)=\{\operatorname{dist}(y,L):L\in A\}.
\end{align*}
Then the following hold:
\begin{enumerate}[label=\textup{(\roman*)}]
\item
If $\dim_HA+\dim_HB>2d-1$, then for every $0<\e<\dim_HA+\dim_HB-(2d-1)$, there exists a Frostman probability measure $\lambda_\e$ on $B$ such that, for $\lambda_\e$-a.e. $y$,
\begin{align*}
\calL^1\left(\left\{t>0:\dim_H\{L\in A:\operatorname{dist}(y,L)=t\}\geq \dim_HA+\dim_HB-(2d-1)-\e\right\}\right)>0.
\end{align*}

\item If $\dim_HA+\dim_HB>2d$, then $\Delta_{\mathrm{line}}^y(A)$ contains a nonempty open interval for some $y\in B$.
\end{enumerate}
\end{corollary}

\begin{corollary}[Multiparameter pinned distances]
Let $p\geq 2$, let $d_j\geq 2$, and set
\[
D=\sum_{j=1}^p d_j,
\qquad
m=\min_{1\leq j\leq p}(d_j-1).
\]
Write
\[
x=(x^{(1)},\ldots,x^{(p)})
\in\R^{d_1}\times\cdots\times\R^{d_p}.
\]
Let $A,B\subset\R^D$ be compact and assume that
\[
x^{(j)}\neq y^{(j)}
\quad\text{for all }(x,y)\in A\times B
\text{ and }1\leq j\leq p.
\]
Define
\[
\Phi_{p}(x,y)
=
\bigl(
|x^{(1)}-y^{(1)}|,\ldots,
|x^{(p)}-y^{(p)}|
\bigr),
\]
and, for each $y\in\R^D$, put
\[
\Delta_{p}^y(A)
=
\{\Phi_{p}(x,y):x\in A\}.
\]
Then the following hold.
\begin{enumerate}[label=\textup{(\roman*)}]
\item
If $\dim_H A+\dim_H B>2D-m$, put
\[
\kappa=\dim_H A+\dim_H B-(2D-m).
\]
Then, for every $0<\e<\kappa$, there exists
a Frostman probability measure $\lambda_\e$
supported on $B$ such that, for
$\lambda_\e$-a.e.\ $y$,
\[
\calL^p\Bigl(
\Bigl\{\bt\in(0,\infty)^p:
\dim_H\bigl\{
x\in A:
\Phi_{p}(x,y)=\bt
\bigr\}
\geq\kappa-\e
\Bigr\}
\Bigr)>0.
\]

\item
If $\dim_H A+\dim_H B>2D-m+p$, then
$\Delta_{p}^y(A)$
has nonempty interior for some $y\in B$.
\end{enumerate}
\end{corollary}

\begin{remark}
Greenleaf, Iosevich, and Taylor
\cite{GreenleafIosevichTaylor2021}
obtained nonempty interior for unpinned multiparameter
distance sets under the condition
$\dim_H A+\dim_H B>2D-m$.

For pinned multiparameter distances, Du, Ou, and Zhang
\cite{DuOuZhang2022} proved that if
$E\subset\R^D$ is compact and
\[
\dim_H E>
\begin{cases}
D-\dfrac{d_{\min}}2+\dfrac14,
&d_{\min}\text{ even},\\[2mm]
D-\dfrac{d_{\min}}2+\dfrac14+
\dfrac1{4d_{\min}},
&d_{\min}\text{ odd},
\end{cases}
\]
where $d_{\min}=\min_j d_j$, then there exists
$y\in E$ such that the pinned multiparameter distance
set has positive $p$-dimensional Lebesgue measure.

While \cite{DuOuZhang2022} obtain a sharper threshold in the
one set setting, our corollary allows the source set $A$ and the set of
pins $B$ to have different Hausdorff dimensions and
provides stronger geometric conclusions under different
hypotheses. To illustrate this, take $p=2$ and
$d_1=d_2=6$, so that $D=12$ and $m=5$.
The threshold of \cite{DuOuZhang2022} is
\[
\dim_H E>12-\frac62+\frac14=9.25.
\]
In contrast, if $\dim_H B=11.9$, our corollary gives:
\begin{enumerate}[label=\textup{(\roman*)}]
\item If $\dim_H A>7.1$, then for almost every pin
$y\in B$, with respect to a suitable Frostman measure,
the corresponding fibers in $A$ have positive Hausdorff
dimension for a positive Lebesgue measure set of
distance parameters.

\item If $\dim_H A>9.1$, then there exists $y\in B$
such that $\Delta_{p}^y(A)$
has nonempty interior. These are asymmetric conclusions with a separate pin set of dimension $11.9$. They do not improve the one-set threshold of
\cite{DuOuZhang2022}, as the hypotheses, as well as the
conclusions, are different.
\end{enumerate}
\end{remark}

The rest of the paper is organized as follows. In Section~\ref{sec: background}, we recall some background on Fourier integral
operators and dimensions of sets used throughout
the paper. In Section~\ref{sec: pinned-2points_configurations},
we prove the pinned configuration theorem. Section~4 treats the nonlinear
projection results, including criteria for failure of Peres--Schlag
transversality and dimension estimates under rank conditions. The slicing theorem is proved in Section~\ref{sec: slicing_sets}. 
Finally, Section~\ref{sec: projection_nontranversal_interior}
applies these results to Harris's restricted projection family.

\section{Background on Fourier integral operators and dimensions of sets}\label{sec: background}
\subsection{Fourier integral operators}
Let $X$ be a $C^\infty$ manifold of dimension $d$. The cotangent bundle $T^\ast X$ is a \textit{symplectic manifold} with respect to the \textit{canonical symplectic two-form}, $\omega=\sum dx_j\wedge d\xi_j$ (with respect to any local coordinates). The zero section  $\left\{ \xi=0\right\}$ of $T^\ast X$ will be denoted by $\bzero$. A submanifold $\Lambda\subset T^\ast X\setminus \bzero$ is called \textit{conic} if it is invariant under $(x,\xi)\mapsto (x,\tau\xi)$ for $0<\tau<\infty$. If $\Lambda$ is smooth, conic, $\dim (\Lambda)=d=\frac{1}{2}\dim(T^\ast X)$, and $\omega\vert_\Lambda\equiv 0$, then we call $\Lambda $ a \textit{conic Lagrangian submanifold} of $T^\ast X$.

Assume that $X$ and $Y$ are smooth manifolds of dimensions $d_X$, $d_Y$, respectively. The cotangent bundles $T^\ast X$, $T^\ast Y$ are symplectic manifolds, with canonical symplectic two-forms $\omega_{T^\ast X},\omega_{T^\ast Y}$, respectively. Then $\omega_{T^\ast X}-\omega_{T^\ast Y}=\sum dx_j\wedge d\xi_j-\sum dy_j\wedge d\eta_j$ is the difference symplectic form on $T^\ast X\times T^\ast Y$. A \textit{canonical relation} is a submanifold $\Lambda\subset (T^\ast X\setminus \bzero)\times (T^\ast Y\setminus 0)$ which is conic Lagrangian with respect to $\omega_{T^\ast X}-\omega_{T^\ast Y}$.

Let $p\geq  1$ be an integer and let $\phi: X\times Y\times (\R^p\setminus \bzero)\to \R$ be a smooth phase function that is positively homogeneous of degree $1$ in $\theta\in \R^p$, i.e., $\phi(x,y,\tau\theta)=\tau\cdot \phi(x,y,\theta)$, for all $\tau\in \R_{+}$. Let $\Sigma_\phi$ be the \textit{critical set} of $\phi$ in the $\theta$ variables,
\begin{align*}
    \Sigma_\phi:=\left\{  (x,y,\theta)\in X\times Y\times (\R^p\setminus \bzero): d_\theta\phi(x,y,\theta)=0\right\}.
\end{align*}
The phase function $\phi(x,y,\theta)$ is said to be \textit{nondegenerate} if
\begin{align*}
    d_x\phi(x,y,\theta)\neq 0\quad \text{ and }\quad d_y\phi(x,y,\theta)\neq 0,\quad \forall (x,y,\theta)\in \Sigma_\phi,
\end{align*}
and 
\begin{align*}
    \rank\left[d_{x,y,\theta}d_\theta\phi(x,y,\theta)\right]=p,\quad \forall(x,y,\theta)\in \Sigma_\phi.
\end{align*}
This means that the $p$ differentials $\left\{ d(d_{\theta_l}\phi)\right\}_{l=1}^p$ are linearly independent, which implies that $\Sigma_\phi$ is a $C^\infty$ conic submanifold of $X\times Y\times (\R^p\setminus 0)$ with codimension $p$. Moreover, one can check that the map
\begin{align*}
    F:\Sigma_\phi\ni (x,y,\theta) \to (x,d_x\phi(x,y,\theta); y,-d_y\phi(x,y,\theta))\in (T^\ast X\setminus\bzero)\times (T^\ast Y\setminus \bzero)
\end{align*}
is an immersion, whose image $C_\phi:=F(\Sigma_\phi)$ is an immersed canonical relation; the phase function $\phi$ is said to \textit{parametrize} $C_\phi$.

Let $\Lambda\subset (T^\ast X\setminus \bzero)\times (T^\ast Y\setminus \bzero)$ be a canonical relation and $m\in \R$. We define $I^m(X,Y;\Lambda)$, the class of \textit{Fourier integral operators} $A:\mathcal{E}'(X)\to \mathcal{D}'(Y)$ of order $m$, as the collection of operators whose Schwartz kernels are locally finite sums of oscillatory integrals of the form
\begin{align*}
    K_A(x,y)=\int_{\R^p}e^{i\phi(x,y,\theta)}a(x,y,\theta)\,d\theta,
\end{align*}
where $\phi$ is a nondegenerate phase function as above parametrizing some relatively open $C_\phi\subset \Lambda$, and the amplitude $a(x,y,\theta)$ is a symbol of order $\gamma=m-\frac{p}{2}+\frac{d_X+d_Y}{4}$.

For a general canonical relation $\Lambda\subset (T^\ast X\setminus \bzero)\times (T^\ast Y\setminus \bzero)$, consider the natural projections $\pi_L: T^\ast X\times T^\ast Y\to T^\ast X$, and $\pi_R: T^\ast X\times T^\ast Y\to T^\ast Y$ restricted to $\Lambda$. We use the same notation for the restricted maps. One can check that, at any point $c_0\in \Lambda$, one has $\corank(D\pi_L)(c_0)=\corank(D\pi_R)(c_0)$. The canonical relation $\Lambda$ is said to be \textit{nondegenerate} if this $\corank(D\pi_L)=0$ at all points of $\Lambda$, that is, if the differentials $D\pi_L$ and $D\pi_R$ have maximal rank. If $\dim X=\dim Y$, then $\Lambda$ is nondegenerate if and only if $\pi_L,\pi_R$ are local diffeomorphisms, and then $\Lambda$ is a local \textit{canonical graph}, i.e., locally near any $c_0\in \Lambda$, one has $\Lambda$ equal to the graph of a canonical transformation. If $\dim X> \dim Y$, then the condition that $\Lambda$ is nondegenerate is equivalent to $\pi_L$ being an immersion and $\pi_R$ being a submersion. 

We will write the order of the FIO $A\in I^m(X,Y;\Lambda)$ as $m=m_{\text{eff}}-\frac{\vert d_X-d_Y\vert}{4},$
where the effective order of $A$ is defined as
\begin{align*}
    m_{\text{eff}} := m+\frac{\vert d_X-d_Y\vert}{4}.
\end{align*}
For $s\in \R$, define the \textit{$L^2$-based Sobolev space} of 
order $s$ on $\R^d$ to be the space of all tempered distributions $f$ such that the distribution $\widehat{f}(\xi)(1+\vert\xi\vert^2)^{s/2}$ 
lies in $L^2$, and give this space the norm
\begin{align*}
    \left\Vert f\right\Vert_{L^2_s(\R^d)}:= \left(\int_{\R^d} |\widehat{f}(\xi)|^2\langle \xi\rangle^{2s}\,d\xi\right)^{1/2}, \quad \langle\xi\rangle=(1+|\xi|^2)^{1/2}.
\end{align*}
The following are the standard $L^2$-Sobolev estimates for FIOs.
\begin{samepage}
\begin{theorem}[H{\"o}rmander \cite{Hormander1971, Hormanderbook}]\label{thm: FIO_standard}
    Suppose that $\Lambda\subset (T^\ast X\setminus \bzero)\times (T^\ast Y\setminus \bzero)$ is a canonical relation, where $\dim(X)=d_X$ and $\dim(Y)=d_Y$, and $A\in I^{m_{\text{eff}}-\frac{\vert d_X-d_Y\vert}{4}}(X,Y;\Lambda)$ has a compactly supported Schwartz kernel.
    \begin{itemize}
        \item[(i)] If $\Lambda$ is nondegenerate, then $A: L^2_s(X)\to L^2_{s-m_{\text{eff}}}(Y)$ for all $s\in \R$. Furthermore, the operator norm depends boundedly on a finite number of derivatives of the amplitude and phase function.
        \item[(ii)] If the spatial projections from $\Lambda$ to $X$ and to $Y$ are submersions and, for some $q$, the corank of $D\pi_L$ (and thus that of $D\pi_R$) is $\leq  q$ at all points of $\Lambda$, then $A: L^2_s(X)\to L^2_{s-m_{\text{eff}}-\frac{q}{2}}(Y)$ for all $s\in \R$. 
    \end{itemize}
\end{theorem}
\end{samepage}
For classes of canonical relations for which $\pi_L$ and/or $\pi_R$ degenerate in specific ways, the loss of derivatives is often less than $q/2$. We recall the following result due to Greenleaf and Seeger \cite{Greenleaf_Seeger_1994}, for which the authors proved a $1/4$-derivative loss when one cotangent projection has Whitney fold singularities.

\begin{theorem}[\cite{Greenleaf_Seeger_1994}]\label{thm:one-sided-fold}
   Suppose $d_X=d_Y$. If one of the cotangent projections of $\Lambda$ is everywhere either a local diffeomorphism or a Whitney fold, and $A\in I^m(X,Y;\Lambda)$ has a compactly supported Schwartz kernel, then
\begin{align*}
A:L_s^2(X)\longrightarrow L_{s-m-1/4}^2(Y),
\qquad s\in\R.
\end{align*}
\end{theorem}

\subsection{Dimensions of sets}
We recall Frostman's lemma and the energy criterion for Hausdorff dimension.
For a set $A\subset\R^d$, let $\mathcal M(A)$ denote the
nonzero finite positive Borel measures with compact support contained
in $A$.
\begin{lemma}[Frostman's lemma, Theorem 2.7, \cite{Mattila2015}]
Let $0\leq  s\leq  d$.
For a Borel set $A\subset \mathbb{R}^d$, the $s$-dimensional Hausdorff measure of $A$ is positive if and only if there exists a probability measure $\mu\in\mathcal{M}(A)$ satisfying
\begin{align}\label{eq: frostman}
    \mu(B(x,r))\lesssim r^s, \quad \forall\, x\in \mathbb{R}^d,\,\,r>0.
\end{align}
In particular,
\begin{align*}
    \dim_{H} A
    &= \sup\left\{  s\leq  d: \exists \mu\in \mathcal{M}(A) \text{ such that } \eqref{eq: frostman} \text{ holds}\right\}.
\end{align*}
\end{lemma}
Frostman's lemma yields that for any exponent $0<s< \dim_H (A)$, there exists a probability measure $\mu$ on $A$ satisfying \eqref{eq: frostman}. 
For $0<s<d$, the \textit{$s$-energy integral} of a measure $\mu\in \mathcal{M}(\R^d)$ (see \cite{Mattila2015})
is
\begin{align*}
    I_s(\mu)\coloneqq\iint\vert x-y\vert^{-s}\,d\mu(x)d\mu(y) = c(d,s)\int\vert\widehat{\mu}(\xi)\vert^2\vert\xi\vert^{s-d}d\xi.
\end{align*}
If $\mu$ is a Frostman measure satisfying $\eqref{eq: frostman}$, then $I_t(\mu)<\infty$ for all $t\in (0,s)$. 
Since $\widehat\mu$ is bounded, this implies
\begin{align*}
    \mu\in L_{(t-d)/2}^2(\R^d).
\end{align*}

\section{Pinned \texorpdfstring{$2$}{2}-point configuration sets}\label{sec: pinned-2points_configurations}
Let $d_X,d_Y\geq  2$, and $p\geq  1$. Let $X\subset \R^{d_X}$, $Y\subset \R^{d_Y}$ be open sets, and assume $\Phi\in C^\infty(X\times Y, \R^p)$ satisfies
\begin{align*}
    \rank(D_x\Phi)=p, \quad \rank(D_y\Phi)=p, \quad\quad \forall(x,y)\in X\times Y.
\end{align*}
Let $\mu\in \mathcal{M}(X)$, $\lambda\in \mathcal{M}(Y)$ be probability measures on $X$ and $Y$, respectively.
Let $\chi\in C_0^\infty(X\times Y)$ be a compactly supported smooth function such that $\chi\geq 0$ and $\chi=1$ in a neighborhood of $\supp\mu\times\supp\lambda$. Let $W \subset\R^p$ be a bounded open set such that $\Phi(\supp (\chi))\subset W$.
For each $y\in Y$, define a measure $\nu_y$ by 
\begin{align*}
    \int_{\R^p}\varphi(\bt)d\nu_y(\bt)=\int_X\chi(x,y)\varphi(\Phi(x,y))d\mu(x).
\end{align*}
In other words, $\nu_y$ is the pushforward measure of $\chi(\cdot,y)\mu$ on $X$ under the map $x\mapsto \Phi(x,y)$. 

Recall that the family of generalized Radon transforms $\calR_\bt$ associated with $\Phi$ is defined as follows. For each $\bt\in \R^p$, for all $f\in C_0^\infty(X)$,
\begin{align*}
    \calR_\bt f(y)=\int_{\{\Phi(x,y)=\bt\}}f(x)\chi(x,y)d\sigma_{y,\bt}(x), \quad y\in Y
\end{align*}
where $\sigma_{y,\bt}$ is the induced Leray measure of $\{x\in X: \Phi(x,y)=\bt\}$.  Note that in the sense of distributions, $\nu_y(\bt)=\calR_\bt\mu(y)$, for $\bt\in W$, $y\in Y$.

We discuss some properties of the family $\{\calR_\bt\}_{\bt\in W}$. For each $\bt\in W$, the operator $\calR_\bt$ is a Fourier integral operator with Schwartz kernel
\begin{align*}
    K_\bt(x,y)=\int_{\R^p}e^{2\pi i \tau \cdot (\Phi(x,y)-\bt)}\chi(x,y) d\tau
\end{align*}
where $\chi(x,y)$ is a classical symbol of order $0$. The operator $\calR_\bt$ has order
\begin{align*}
    m=0+\frac{p}{2}-\frac{d_X+d_Y}{4},
\end{align*}
and effective order $m_{\mathrm{eff}}=m+\frac{|d_X-d_Y|}{4}=\frac{p}{2}-\frac{\min\{d_X,d_Y\}}{2}$.

The canonical relation associated to $\calR_\bt$ is
\begin{align*}
    \Lambda_\bt:=\left\{\left(x, \sum_{j=1}^p d_x\Phi_j(x,y)\tau_j; y,-\sum_{j=1}^pd_y\Phi_j(x,y)\tau_j\right): (x,y)\in Z_\bt,\tau\in \R^p\setminus \bzero\right\},
\end{align*}
where the incidence relation $Z_\bt$ is defined by
$Z_\bt=\{(x,y)\in X\times Y: \Phi(x,y)=\bt\}$.
We assume that there is a smoothing order $\alpha_\Phi$, $0\leq  \alpha_\Phi\leq \frac{\min\{d_X,d_Y\}-p}{2}$, determined by the geometry of the canonical relations $\Lambda_\bt$, such that, 
for every $s\in \R$, one has, 
\begin{align}\label{eq: smoothing_condition}
     \calR_\bt: L^2_s(X)\to L^2_{s+\alpha_\Phi}(Y),
\end{align}
uniformly for $\bt\in W$.

\begin{samepage}
Theorem \ref{thm: main_pinned_configuration} follows from the following proposition.
\begin{proposition}\label{prop: pinned_configuration}
Suppose that the operators $\calR_\bt$ satisfy \eqref{eq: smoothing_condition} with some smoothing order $\alpha_\Phi$, $0\leq  \alpha_\Phi\leq \frac{\min\{d_X,d_Y\}-p}{2}$. Assume that $\mu$ and $\lambda$ are Frostman measures on $X$ and $Y$ with Frostman exponents $s_\mu$ and $s_\lambda$, such that $I_{s_\mu}(\mu)<\infty$ and $I_{s_\lambda}(\lambda)<\infty$, respectively. Set
    \begin{align*}
        \eta=\alpha_\Phi+\frac{s_\mu-d_X}{2}+\frac{s_\lambda-d_Y}{2}.
    \end{align*}
    Then the following statements hold.
\begin{itemize}
    \item[(i)] If $\eta>0$, then for $\lambda$-a.e. $y$, $\nu_y$ has an $L^2$ density and therefore $\calL^p(\supp(\nu_y))>0$.
    \item[(ii)] If $\eta>\frac{p}{2}$, then for $\lambda$-a.e. $y$, $\nu_y$ has a continuous density and therefore $\supp(\nu_y)$ has nonempty interior.
\end{itemize} 
\end{proposition}
\end{samepage}
Next, in each ambient space $\R^n$ ($n=d_X,d_Y$), we choose a smooth cutoff function $\phi\in C_0^\infty(\R^{n})$ such that $\phi\equiv 1$ on $B(0,1)$, and $\supp \phi\subset B(0,2)$. Put $\psi(\xi)=\phi(\xi)-\phi(2\xi)$, for $\xi\in \R^{n}$, then $\psi$ is supported in $\{1/2\leq  |\xi|\leq  2\}$. Put $\psi_0=\phi$, and $\psi_j=\psi(2^{-j}\xi)$, for all $j\geq  1$, then we have
\begin{align*}
    \psi_0(\xi)+\sum_{j=1}^\infty \psi_j(\xi)=1, \quad \supp(\psi_j)\subset \{2^{j-1}\leq  |\xi|\leq  2^{j+1}\}, \forall j\geq  1.
\end{align*}
For $f\in\mathcal S'(\R^n)$, for each $j\geq 0$, denote $f_j=P_jf$ as the $j$-th Littlewood--Paley projection of $f$, i.e.,
\begin{align*}
    \widehat{P_jf}(\xi)=\hat{f}(\xi)\psi_j(\xi), \quad \xi\in \R^n.
\end{align*}
We will use superscripts $P^X_j$, $P^Y_j$ to distinguish the ambient spaces when necessary. From the bounded overlap of the annuli and Plancherel, one has, for every $s\in \R$,
\begin{align*}
    \Vert f\Vert_{L_s^2}^2\approx \sum_{j\geq  0}2^{2js}\Vert P_jf\Vert_{L^2}^2.
\end{align*}
We use the Frostman estimate and the input--output frequency argument of \cite[Lemmas 2.1 and 2.3]{Palsson_Psaromiligkos_2026}.
The latter extends to the present vector-valued setting.

\begin{lemma}\label{lem: Frostman_LP}
    Let $0<s<n$. Let $\mu\in \mathcal{M}(\R^n)$ be an $s$-Frostman measure, so that
        \begin{align*}
            \mu(B(z,r))\leq  C_\mu r^s, \quad \forall z\in \R^n, r>0.
        \end{align*}
    \begin{itemize}
        \item[(i)] If $I_s(\mu)<\infty$, then for each $j\geq  0$,
        \begin{align*}
            \Vert P_j\mu\Vert_{L^2(\R^{n})}\lesssim 2^{j(n-s)/2}I_s(\mu)^{1/2},
        \end{align*}
        where $P_j\mu$ is the $j$-th Littlewood--Paley projection of $\mu$ in $\R^n$.
        \item[(ii)] If $\widehat{g}$ is supported where $|\xi|\lesssim 2^j$, then
        \begin{align*}
            \Vert g\Vert_{L^2(d\mu)}\lesssim C_\mu^{1/2}2^{j(n-s)/2}\Vert g\Vert_{L^2(\R^n)}.
        \end{align*}
    \end{itemize}
\end{lemma}

\begin{lemma}\label{lem: input_output_frequency}
    Let $P_j^X$ and $P_i^Y$ denote Littlewood--Paley projections in the input and output variables. There exists an integer $C_0\geq  1$ such that, for every $M\geq  0$,
    \begin{align*}
       \Vert P_i^Y\calR_\bt P_j^X\Vert_{L^2(X)\to L^2(Y)}\lesssim_M 2^{-M\max\{i,j\}},
    \end{align*}
    whenever $|i-j|>C_0$, uniformly for $\bt\in W$. 
\end{lemma}
We are now ready to prove Proposition \ref{prop: pinned_configuration}.
\begin{proof}[Proof of Proposition \ref{prop: pinned_configuration}]
    Fix $0\leq  \beta<\eta$.
    Our goal is to show that,
    \begin{align}\label{eq: nu_y_average_energy}
        \int_Y\Vert\nu_y\Vert_{L_{\beta}^2(\R^p)}^2d\lambda(y)<\infty,
    \end{align}
    where
    \begin{align*}
        \Vert g\Vert_{L_{\beta}^2(\R^p)}^2=\int_{\R^p}|\widehat{g}(\tau)|^2\langle \tau\rangle^{2\beta}d\tau.
    \end{align*}
    To prove \eqref{eq: nu_y_average_energy}, write $\mu=\sum_{j\geq  0}\mu_j$, where for each $j$, $\mu_j=P^X_j\mu$ is the $j$-th Littlewood--Paley projection of $\mu$. For each $j\geq  0$, $y\in Y$, and $\bt\in \R^p$, put $g_j(y,\bt)=\calR_\bt \mu_j(y)$. Note that these functions vanish outside the fixed compact set $\Phi(\supp\chi)\subset W$.
    
    By linearity of $\calR_\bt$,
    \begin{align*}
        \nu_y(\bt)=\calR_\bt\mu(y)=\sum_{j\ge 0}\calR_\bt\mu_j(y)=\sum_{j\geq  0} g_j(y,\bt),\quad\text{in }\mathcal D'(\mathbb R^p_{\bt}),
\text{ for each }y.
    \end{align*}
    For every finite $N$, the triangle inequality gives
    \begin{align*}
        \Vert \sum_{j=0}^{N} g_j\Vert_{L^2(d\lambda;L_{\beta}^2)}\leq  \sum_{j=0}^{N}\Vert g_j\Vert_{L^2(d\lambda;L_{\beta}^2)}.
    \end{align*}
    We pass to the limit after proving that the right hand series converges.
    So it suffices to bound each $\Vert g_j\Vert_{L^2(d\lambda; L_{\beta}^2)}$ and sum over $j$.
    
    Fix $j\geq  0$. To bound
\begin{align*}
\Vert g_j\Vert_{L^2(d\lambda;L_\beta^2)}^2&=
\int_Y\int_{\R^p}|\widehat g_j(y,\tau)|^2\langle\tau\rangle^{2\beta}\,d\tau\,d\lambda(y),
\end{align*}
we need to split the integral into two parts, low $\bt$-frequencies and high $\bt$-frequencies. The low $\bt$-frequencies can be controlled by an $L^2$ estimate $\Vert g_j\Vert_{L^2(d\lambda\,d\bt)}$.
To handle high frequencies, we use the rank condition and integration by parts to obtain a decay estimate for $|\widehat{g}_j(y,\tau)|$. 

First, we will give a bound for $\Vert g_j\Vert_{L^2(d\lambda\,d\bt)}$. 
For each $\bt\in W$, by Lemma \ref{lem: Frostman_LP}(ii),  one has
\begin{align*}
\Vert g_j(\cdot,\bt)\Vert_{L^2(d\lambda)}&\leq 
\sum_{i\geq 0}\Vert P_i^Y\calR_\bt\mu_j\Vert_{L^2(d\lambda)}
\\
&\lesssim\sum_{i\geq 0}2^{i\frac{d_Y-s_\lambda}{2}}\Vert P_i^Y\calR_\bt\mu_j\Vert_{L^2}\\
&\lesssim  \sum_{|i-j|\leq  C_0}2^{i\frac{d_Y-s_\lambda}{2}}\Vert P_i^Y\calR_\bt\mu_j\Vert_{L^2}+\sum_{|i-j|> C_0}2^{i\frac{d_Y-s_\lambda}{2}}\Vert P_i^Y\calR_\bt\mu_j\Vert_{L^2},
\end{align*}
where $C_0$ comes from Lemma \ref{lem: input_output_frequency}.
For the first sum, applying the smoothing assumption \eqref{eq: smoothing_condition} with $s=0$ gives
\begin{align*}
\Vert P_i^Y\calR_\bt\mu_j\Vert_{L^2}\lesssim 2^{-i\alpha_\Phi}\Vert\calR_\bt\mu_j\Vert_{L_{\alpha_\Phi}^2}
\lesssim 2^{-i\alpha_\Phi}\Vert\mu_j\Vert_{L^2}.
\end{align*}
Summing over $|i-j|\leq  C_0$, this implies that
\begin{align}\label{eq: g_j_L^2_lambda_near_diagonal}
    \sum_{|i-j|\leq  C_0} 2^{i\frac{d_Y-s_\lambda}{2}}\Vert P_i^Y\calR_\bt\mu_j\Vert_{L^2}
\lesssim
2^{j(\frac{d_Y-s_\lambda}{2}-\alpha_\Phi)}\Vert\mu_j\Vert_{L^2}.
\end{align}
For the second sum, for each $i$ such that $|i-j|>C_0$, the enlarged cutoff version of Lemma \ref{lem: input_output_frequency} yields that
\begin{align*}
\Vert P_i^Y\calR_\bt\mu_j\Vert_{L^2} \lesssim_M 2^{-M\max\{i,j\}}\Vert\mu_j\Vert_{L^2}.
\end{align*}
Choosing $M>\frac{d_Y-s_\lambda}{2}+\alpha_\Phi+2$, we obtain
\begin{align}\label{eq: g_j_L^2_lambda_off_diagonal}
\sum_{|i-j|>C_0} 2^{i\frac{d_Y-s_\lambda}{2}}\Vert P_i^Y\calR_\bt\mu_j\Vert_{L^2}
&\lesssim  2^{-j(\alpha_\Phi+1)}\Vert\mu_j\Vert_{L^2}\lesssim 2^{j(\frac{d_Y-s_\lambda}{2}-\alpha_\Phi)}\Vert\mu_j\Vert_{L^2}.
\end{align}
Combining the two estimates \eqref{eq: g_j_L^2_lambda_near_diagonal} and \eqref{eq: g_j_L^2_lambda_off_diagonal} gives
\begin{align*}
\Vert g_j(\cdot,\bt)\Vert_{L^2(d\lambda)} \lesssim 2^{j(\frac{d_Y-s_\lambda}{2}-\alpha_\Phi)}\Vert\mu_j\Vert_{L^2},
\end{align*}
uniformly in $\bt$. Since $g_j$ vanishes outside the fixed bounded set $\Phi(\supp\chi)\subset W$, one obtains
\begin{align}\label{eq: g_j_L2_bound}
\Vert g_j\Vert_{L^2(d\lambda\,d\bt)}
\lesssim 2^{j(\frac{d_Y-s_\lambda}{2}-\alpha_\Phi)}\Vert\mu_j\Vert_{L^2}.
\end{align} 
Next, we give a bound for $|\widehat{g}_j(y,\tau)|$. Write
\begin{align}\label{eq: g_j_Fourier_transform}
\widehat g_j(y,\tau) &=
\int_{\R^p}e^{-2\pi i\bt\cdot\tau} \calR_\bt\mu_j(y)\,d\bt =\int_{\R^{d_X}} K(y,\xi,\tau)\widehat{\mu_j}(\xi)\,d\xi,
\end{align}
where 
\begin{align*}
    K(y,\xi,\tau)=\int_X e^{2\pi i(x\cdot\xi-\tau\cdot\Phi(x,y))} \chi(x,y)\,dx.
\end{align*}
We note that $\tau\in \R^p$ is dual to the configuration variable $\bt$, whereas $\xi\in \R^{d_X}$ is dual to the variable $x$.

We consider the phase function $\phi(x)=x\cdot\xi-\tau\cdot\Phi(x,y)$, $x\in \R^{d_X}$. By assumption, $\rank D_x\Phi=p$, and hence the transpose $(D_x\Phi(x,y))^T:\R^p\to\R^{d_X}$ is injective.
Thus the continuous function
\begin{align*}
(x,y,\omega)\longmapsto |(D_x\Phi(x,y))^T\omega|
\end{align*}
is strictly positive on the compact set $\supp\chi\times S^{p-1}$.
Consequently, there exists $c>0$ such that
\begin{align*}
|(D_x\Phi(x,y))^T\tau|\geq  c|\tau|, \qquad (x,y)\in\supp\chi,\quad \tau\in\R^p.
\end{align*}
Choose a fixed constant $C_1\geq \max\{1,4/c\}$. If $|\tau|>C_12^j$, then $c|\tau|\geq \frac{4}{C_1}|\tau|\geq   2^{j+2}$.

Note that the Fourier support of $\mu_j$ is contained in $\{|\xi|\leq 2^{j+1}\}$. The gradient of $\phi$ with respect to $x$ therefore satisfies
\begin{align*}
\left| \nabla_x\bigl(x\cdot\xi-\tau\cdot\Phi(x,y)\bigr) \right|
&=|\xi-(D_x\Phi(x,y))^T\tau|\geq  c|\tau|-|\xi|\geq  c|\tau|-2^{j+1}\geq  \frac c2|\tau|.
\end{align*}
Thus $\phi(x)$ has no stationary points on $\supp\chi$ in this frequency region. 
Therefore, integrating by parts, one has, for every integer $N\geq 1$,
\begin{align*}
|K(y,\xi,\tau)| \lesssim_N |\tau|^{-N}, \qquad |\tau|>C_12^j,
\end{align*}
uniformly in $y$ and in $\xi$ on the Fourier support of $\mu_j$.

Thus, from \eqref{eq: g_j_Fourier_transform}, by Cauchy--Schwarz and Plancherel, it follows that
\begin{align}
|\widehat g_j(y,\tau)| &\lesssim_N |\tau|^{-N} \int_{|\xi|\leq 2^{j+1}} |\widehat{\mu_j}(\xi)|\,d\xi\notag\\
&\lesssim_N |\tau|^{-N} \big|\{\xi:|\xi|\leq 2^{j+1}\}\big|^{1/2}
\Vert\widehat{\mu_j}\Vert_{L^2}\label{eq: g_j_hat_pointwise_bound_high_frequency}\\
&\lesssim_N 2^{jd_X/2}|\tau|^{-N}\Vert\mu_j\Vert_{L^2}, \qquad |\tau|>C_12^j.\notag
\end{align}
We are now ready to estimate $\Vert g_j\Vert_{L^2(d\lambda;L_\beta^2)}^2$. Split the squared Sobolev norm as
\begin{align*}
\Vert g_j\Vert_{L^2(d\lambda;L_\beta^2)}^2 &= \int_Y\int_{\R^p}
|\widehat g_j(y,\tau)|^2 \langle\tau\rangle^{2\beta}\,d\tau\,d\lambda(y)=J_{\mathrm{low}}+J_{\mathrm{high}},
\end{align*}
where the two terms correspond to $|\tau|\leq  C_12^j$ and $|\tau|>C_12^j$, respectively.

On the low-frequency region, $\langle\tau\rangle^{2\beta}\lesssim2^{2j\beta}$ because $\beta\geq 0$. Thus Plancherel in $\bt$ and the preceding $L^2$ estimate \eqref{eq: g_j_L2_bound} imply
\begin{align}
J_{\mathrm{low}} &\lesssim 2^{2j\beta} \int_Y\int_{\R^p}
|\widehat g_j(y,\tau)|^2\,d\tau\,d\lambda(y)\notag\\
&= 2^{2j\beta}\Vert g_j\Vert_{L^2(d\lambda\,d\bt)}^2\label{eq: low_frequencies_estimate}\\
&\lesssim 2^{2j(\beta+\frac{d_Y-s_\lambda}{2}-\alpha_\Phi)}\Vert\mu_j\Vert_{L^2}^2.\notag
\end{align}
On the high-frequency region, we apply the pointwise decay estimate \eqref{eq: g_j_hat_pointwise_bound_high_frequency} above and obtain
\begin{align}
J_{\mathrm{high}}
&\lesssim_N 2^{jd_X}\Vert\mu_j\Vert_{L^2}^2 \int_{|\tau|>C_12^j} |\tau|^{-2N}\langle\tau\rangle^{2\beta}\,d\tau\notag\\
&\lesssim_N 2^{jd_X}\Vert\mu_j\Vert_{L^2}^2 \int_{C_12^j}^{\infty}
\rho^{2\beta-2N+p-1}\,d\rho\label{eq: high_frequencies_estimate}\\
&\lesssim_N 2^{j(d_X+2\beta+p-2N)}\Vert\mu_j\Vert_{L^2}^2\notag\\
&\lesssim_L 2^{-j2L}\Vert\mu_j\Vert_{L^2}^2,\notag
\end{align}
provided $N>L+d_X/2+\beta+p/2$, where $L>\frac{d_X-s_\mu}{2}$. 

Combining the estimates \eqref{eq: low_frequencies_estimate} and \eqref{eq: high_frequencies_estimate} and taking square roots, we obtain
\begin{align*}
\Vert g_j\Vert_{L^2(d\lambda;L^2_\beta)} \lesssim_L \left( 2^{j(\beta+\frac{d_Y-s_\lambda}{2}-\alpha_\Phi)} +2^{-jL} \right) \Vert\mu_j\Vert_{L^2}.
\end{align*}
By Lemma \ref{lem: Frostman_LP} (i),
\begin{align*}
\Vert\mu_j\Vert_{L^2} \lesssim 2^{j(d_X-s_\mu)/2}I_{s_\mu}(\mu)^{1/2}.
\end{align*}
Hence
\begin{align*}
\Vert g_j\Vert_{L^2(d\lambda;L^2_\beta)} \lesssim_L \left( 2^{-j(\eta-\beta)} + 2^{-j(L-\frac{d_X-s_\mu}{2})} \right) I_{s_\mu}(\mu)^{1/2},
\end{align*}
where
\begin{align*}
\eta = \alpha_\Phi +\frac{s_\mu-d_X}{2} +\frac{s_\lambda-d_Y}{2}.
\end{align*}
Since $\beta<\eta$ and $L>(d_X-s_\mu)/2$, summing over $j\geq  0$ gives
\begin{align*}
\sum_{j\ge0} \Vert g_j\Vert_{L^2(d\lambda;L^2_\beta)} <\infty.
\end{align*}
Thus one gets
\begin{align*}
    \int_Y\Vert \nu_y\Vert_{L_{\beta}^2(\R^p)}^2d\lambda(y)<\infty.
\end{align*}
This implies that for $\lambda$-a.e. $y\in Y$, $\Vert\nu_y\Vert_{L^2_\beta}<\infty$.

For part (i), if $\eta>0$, take $\beta=0$. Then for
$\lambda$-a.e. $y$, one has $\nu_y\in L^2(\R^p)$, so $\nu_y$ has an $L^2$ density. Since $\nu_y$ has total mass one, its support has positive Lebesgue measure.

For part (ii), if $\eta>\frac p2$, choose $\frac p2<\beta<\eta$. Then the density of $\nu_y$ belongs to $L_\beta^2(\R^p)$, and the
Sobolev embedding
\begin{align*}
L_\beta^2(\R^p)\hookrightarrow C^0(\R^p)
\end{align*}
shows that it has a continuous representative. Since this density has total mass one, it is positive on a nonempty open set. Therefore $\supp\nu_y$ has nonempty interior.
\end{proof}

\begin{remark}\label{rem:averaged-sobolev}
    In the setting of Proposition~\ref{prop: pinned_configuration}, put
$\eta=\alpha_\Phi+(s_\mu-d_X)/2+(s_\lambda-d_Y)/2$.
The preceding averaged estimate also holds for every real
$\beta<\eta$, including negative $\beta$:
\[
\int_Y\|\nu_y\|_{L_\beta^2(\R^p)}^2\,d\lambda(y)<\infty.
\]
Indeed, for $\beta<0$, split the configuration frequency into
\[
|\tau|\leq c_0 2^j,\qquad
c_0 2^j<|\tau|\leq C_1 2^j,\qquad
|\tau|>C_1 2^j,
\]
where $c_0>0$ is sufficiently small. In the first region,
$|\xi-D_x\Phi(x,y)^T\tau|\gtrsim2^j$ on the Fourier
support of $\mu_j$, so integration by parts gives rapid
decay in $j$. In the middle region,
$\langle\tau\rangle^{2\beta}\lesssim2^{2j\beta}$,
and \eqref{eq: g_j_L2_bound} applies. The last region is
treated by \eqref{eq: high_frequencies_estimate}.
The resulting bounds are summable for $\beta<\eta$,
as in the preceding proof.
\end{remark}

\section{Nonlinear projections: Transversality and rank conditions}
Let
$X\subset\R^{d_X}$ and $Y\subset\R^{d_Y}$ be open,
let $\Phi:X\times Y\to\R^p$ be smooth and satisfy the
submersion condition \eqref{eq: rank_condition} as in preceding sections. We give conditions on
$\Phi$ which imply that the transversality condition of Peres and Schlag
fails. We then compare these conditions with the nondegeneracy of the
associated generalized Radon transforms. The latter can still give
pinned configuration and slicing results.

Let $K\subset X$ be compact and let $Q\subset Y$ be open. For distinct
$x,x'\in K$ and $y\in Q$, set
\begin{equation*}
 \Psi_y(x,x')=\frac{\Phi(x,y)-\Phi(x',y)}{|x-x'|}\in\R^p,
\end{equation*}
and write
\begin{equation*}
 \calG_y(x,x')=\det\bigl(D_y\Psi_y(x,x')D_y\Psi_y(x,x')^T\bigr).
\end{equation*}
We recall the transversality condition of \cite[Definition~7.2]{Peres_Schlag_2000};
see also \cite[Chapter~18]{Mattila2015}.

\begin{definition}\label{def:PS-beta}
Let $0\leq \beta<1$. We say that $Q$ is a \emph{region of transversality
of order $\beta$} for $\Phi$ on $K$ if there is $C_\beta>0$ such that,
for every $y_1,y_2\in Q$ and distinct $x,x'\in K$,
\[
 |\Psi_{y_1}(x,x')|+|\Psi_{y_2}(x,x')|
 \leq  C_\beta|x-x'|^\beta
\]
implies
\begin{equation}\label{eq:PS-transversality}
 \calG_{y_1}(x,x')\geq  C_\beta^2|x-x'|^{2\beta}.
\end{equation}
\end{definition}

We note that $\calG_y(x,x')>0$ if and only if
$D_y\Psi_y(x,x')$ has rank $p$.
Thus the Peres-Schlag transversality condition requires
that, whenever the normalized differences are sufficiently
small, the parameter derivative has full rank, with a
quantitative lower bound on its Gram determinant.
\begin{remark}
Since $\Phi$ is smooth, the regularity assumptions in
\cite[Definition 7.2]{Peres_Schlag_2000} hold after compact
localization, by the mean value theorem. We shall therefore
focus on the transversality condition.
\end{remark}

\subsection{Degenerate collisions}
We now describe two ways in which Peres-Schlag transversality
can fail. The first involves two distinct input points whose
images coincide, but whose parameter derivatives fail to distinguish
them. The second arises when two input points approach one another
and the limiting parameter derivative loses rank. We call these
\textit{off-diagonal }and \textit{infinitesimal degenerate collisions}, respectively,
and show that both imply failure of Peres-Schlag transversality.
\begin{definition}
For distinct $x,x'\in K$ and $y\in Q$, define the
\emph{off-diagonal collision matrix}
\begin{equation*}
 \Cphi(x,x';y)=D_y\Phi(x,y)-D_y\Phi(x',y).
\end{equation*}
We call $(x,x',y)$ an \emph{off-diagonal degenerate collision} if
\begin{equation*}
 \Phi(x,y)=\Phi(x',y),\qquad \rank\Cphi(x,x';y)<p.
\end{equation*}
\end{definition}

The rank condition is equivalent to the existence of
$\tau\in\R^p\setminus\bzero$ for which
\begin{equation*}
 D_y(\tau\cdot\Phi)(x,y)=D_y(\tau\cdot\Phi)(x',y).
\end{equation*}
Thus, at an off-diagonal degenerate collision, the two scalar functions
$y'\mapsto\tau\cdot\Phi(x,y')$ and
$y'\mapsto\tau\cdot\Phi(x',y')$ have the same value and first
derivative at $y$.

\begin{proposition}\label{prop:PS-offdiag}
If $(x,x',y)$ is an off-diagonal degenerate collision, with
$x,x'\in K$ and $y\in Q$, then transversality fails on $K\times Q$
for every order $0\leq \beta<1$.
\end{proposition}

\begin{proof}
We have $\Psi_y(x,x')=0$ and
\[
 D_y\Psi_y(x,x')=\frac{\Cphi(x,x';y)}{|x-x'|}.
\]
Consequently,
\[
 \calG_y(x,x')=|x-x'|^{-2p}
 \det\bigl(\Cphi(x,x';y)\Cphi(x,x';y)^T\bigr)=0.
\]
Taking $y_1=y_2=y$ in Definition~\ref{def:PS-beta}, the
hypothesis holds for every $C_\beta>0$, whereas the conclusion
\eqref{eq:PS-transversality} does not.
\end{proof}

We have considered degenerate collisions between two distinct points
$x,x'\in X$. We now examine what happens as $x'$ approaches $x$.
Fix $x\in X$, $y\in Y$, and a unit vector $v\in\R^{d_X}$.
Taking $x'=x+tv$ with $t>0$, Taylor's theorem gives
\begin{align*}
\Psi_y(x,x+tv)
&=-D_x\Phi(x,y)v+O(t),\\
D_y\Psi_y(x,x+tv)
&=-D_y[D_x\Phi(x,y)v]+O(t).
\end{align*}
Thus, if $D_x\Phi(x,y)v=0$, then
\[
|\Phi(x+tv,y)-\Phi(x,y)|=O(t^2).
\]
In particular, the images of $x$ and $x+tv$ are much closer than
the two input points. If, in addition,
\[
\rank\bigl(D_y[D_x\Phi(x,y)v]\bigr)<p,
\]
then the Gram determinant of $D_y\Psi_y(x,x+tv)$ tends to zero
as $t\to0^+$. This motivates the following definition.

\begin{definition}
For $v\in\R^{d_X}\setminus\bzero$, define the $p\times d_Y$ matrix
\begin{equation*}
 \Tphi(x,y;v)=D_y[D_x\Phi(x,y)v].
\end{equation*}
We call $(x,y,v)$ an \emph{infinitesimal degenerate collision} if
\begin{equation}\label{eq:PS-tangent}
 D_x\Phi(x,y)v=0,\qquad \rank\Tphi(x,y;v)<p.
\end{equation}
Equivalently, there is $\tau\in\R^p\setminus\bzero$ such that
\begin{equation*}
 D_x\Phi(x,y)v=0,\qquad
 D_yD_x(\tau\cdot\Phi)(x,y)v=0.
\end{equation*}
\end{definition}

\begin{proposition}
Suppose $(x_0,y_0,v)$ is an infinitesimal degenerate collision,
where $x_0\in\Int(K)$ and $y_0\in Q$, $v\in \R^{d_X}\setminus\bzero$.
Then, for every neighborhood $U\subset X$ of $x_0$ and every
neighborhood $V\subset Q$ of $y_0$, the family
$\{\Phi(\cdot,y):y\in V\}$ fails Peres-Schlag transversality
of every order $0\leq \beta<1$ on $K\cap U$.
\end{proposition}

\begin{proof}
We may assume that $|v|=1$. Let $U\subset X$ and $V\subset Q$ be neighborhoods of $x_0$ and $y_0$, respectively. Since $x_0\in\Int(K)$, for all sufficiently small $t>0$, we have
\[
x_0, x_t=x_0+tv\in K\cap U.
\]
From the smoothness of
$\Phi$, by Taylor's theorem and \eqref{eq:PS-tangent},
\begin{equation}\label{eq:PS-taylor}
\Psi_{y_0}(x_0,x_t)=O(t),\qquad
D_y\Psi_{y_0}(x_0,x_t)
=-\Tphi(x_0,y_0;v)+O(t).
\end{equation}
Since $\rank\left[\Tphi(x_0,y_0;v)\right]<p$, there exists a unit
covector $\tau\in\R^p$ such that $\Tphi(x_0,y_0;v)^T\tau=0$.
By \eqref{eq:PS-taylor}, one has
\[
\left|(D_y\Psi_{y_0}(x_0,x_t))^T\tau\right|=O(t).
\]
This yields that the smallest singular value of
$D_y\Psi_{y_0}(x_0,x_t)$ is $O(t)$.
Since the remaining singular values are bounded, we obtain
\begin{equation}\label{eq:PS-Gram-collision}
\calG_{y_0}(x_0,x_t)=\det\left(D_y\Psi_{y_0}(x_0,x_t)D_y\Psi_{y_0}(x_0,x_t)^T\right)=O(t^2).
\end{equation}
Now fix $0\leq \beta<1$ and $C_\beta>0$. Since
$t=o(t^\beta)$ and $t^2=o(t^{2\beta})$ as $t\to0^+$, by \eqref{eq:PS-taylor} and \eqref{eq:PS-Gram-collision}, we may choose $t>0$ sufficiently small that
\[
2|\Psi_{y_0}(x_0,x_t)|
\leq  C_\beta t^\beta,
\qquad
\calG_{y_0}(x_0,x_t)
<C_\beta^2t^{2\beta}.
\]
Taking $y_1=y_2=y_0\in V$, we see that the
Peres--Schlag transversality condition fails on $K\cap U$
for the parameter region $V$. Since $U$, $V$, $\beta$,
and $C_\beta$ were arbitrary, the proposition follows.
\end{proof}

\subsection{A criterion for transversality of order zero}
We showed above that off-diagonal and
infinitesimal degenerate collisions imply failure of
Peres-Schlag transversality. We now prove that, on compact
sets, these are the only possible obstructions to transversality
of order zero. We first give an equivalent quantitative
condition in terms of the normalized difference and its
parameter derivative.

For a $p\times d_Y$ matrix $M$ write
\[
 \sigma_{\min}(M)=\min_{|\tau|=1}|M^T\tau|.
\]
Thus $\sigma_{\min}(M)>0$ precisely when $M$ has rank $p$. The
following result characterizes transversality of order zero on compact
parameter regions.

\begin{proposition}\label{prop:PS-quantitative}
Let $Q_0\subset Q$ be compact. Then $\Phi$ satisfies
the Peres-Schlag transversality of order $\beta=0$ on $K\times Q_0$
if and only if there exists a constant $c>0$ such that
\begin{equation}\label{eq:PS-quantitative}
\frac{|\Phi(x,y)-\Phi(x',y)|}{|x-x'|}
+
\frac{\sigma_{\min}(\Cphi(x,x';y))}{|x-x'|}
\geq  c
\end{equation}
for all $x,x'\in K$ with $x\neq x'$ and all $y\in Q_0$.
\end{proposition}

\begin{proof}
Suppose \eqref{eq:PS-quantitative} holds for some $c>0$.
Choose $C_0>0$ such that $C_0\leq  \frac c2$, $C_0^2\leq  \left(\frac c2\right)^{2p}$.
Let $x,x'\in K$, $x\neq x'$, and $y_1,y_2\in Q_0$ satisfy
\[
|\Psi_{y_1}(x,x')|+|\Psi_{y_2}(x,x')|\leq  C_0.
\]
Then $|\Psi_{y_1}(x,x')|\leq  c/2$, so
\eqref{eq:PS-quantitative} implies $\sigma_{\min}(D_y\Psi_{y_1}(x,x'))\geq  \frac c2$.
Since the Gram determinant is the product of the squares
of the $p$ singular values, we obtain
\[
\calG_{y_1}(x,x')
\geq  \left(\frac c2\right)^{2p}
\geq  C_0^2.
\]
Thus $\Phi$ satisfies Peres-Schlag transversality of order $\beta=0$.

Conversely, suppose that $\Phi$ satisfies Peres-Schlag
transversality of order $\beta=0$ with constant $C_0>0$.
By smoothness and compactness, there exists $M\geq 1$ such that
\begin{equation*}
\|D_y\Psi_y(x,x')\|\leq  M
\end{equation*}
for all distinct $x,x'\in K$ and $y\in Q_0$.
For such a triple $(x,x',y)$, if
$|\Psi_y(x,x')|\geq  C_0/2$, then the first term in
\eqref{eq:PS-quantitative} is bounded below by $C_0/2$ and we can choose $c=C_0/2$.
We assume otherwise that $2|\Psi_y(x,x')|<C_0$. Taking $y_1=y_2=y$
in Definition~\ref{def:PS-beta}, we obtain
\[
\calG_y(x,x')\geq  C_0^2.
\]
Let $\sigma_1\geq \cdots\geq \sigma_p=\sigma_{\min}(D_y\Psi_y(x,x'))$ be the singular
values of $D_y\Psi_y(x,x')$. Since $\sigma_j\leq M$ for all $j$, one obtains
\[
C_0^2
\leq  \calG_y(x,x')
=\prod_{j=1}^p\sigma_j^2
\leq  M^{2(p-1)}\sigma_{\min}(D_y\Psi_y(x,x'))^2.
\]
Hence
\[
\sigma_{\min}(D_y\Psi_y(x,x'))
\geq  C_0M^{-(p-1)}.
\]
Combining the two cases, we obtain
\[
|\Psi_y(x,x')|
+\sigma_{\min}(D_y\Psi_y(x,x'))
\geq 
\min\left\{\frac{C_0}{2},C_0M^{-(p-1)}\right\}>0.
\]
Since $D_y\Psi_y(x,x')=\frac{\Cphi(x,x';y)}{|x-x'|}$,
this proves \eqref{eq:PS-quantitative}. This completes the proof of the proposition.
\end{proof}

For each $x\in K$, let
$S_K(x)\subset\mathbb S^{d_X-1}$ be the set of all limiting directions of secants of
$K$ near $x$. More precisely, $v\in S_K(x)$ if there exist sequences $x_j,x'_j\in K$, with $x_j\neq x'_j$,
such that
\begin{equation}\label{eq:PS-secant-limit}
x_j,x'_j\longrightarrow x,
\qquad
\frac{x_j-x'_j}{|x_j-x'_j|}\longrightarrow v.
\end{equation}
In particular, if $x\in\Int(K)$, then $S_K(x)=\mathbb S^{d_X-1}$.

\begin{theorem}\label{thm:PS-dichotomy}
Let $Q_0\subset Q$ be compact, and suppose that $\Phi$ is smooth
in a neighborhood of $K\times Q_0$. Then $\Phi$ fails
Peres-Schlag transversality of order $\beta=0$ on $K\times Q_0$
if and only if one of the following occurs:
\begin{enumerate}[label=\textup{(\roman*)}]
\item There exist distinct points $x,x'\in K$ and $y\in Q_0$
such that $(x,x',y)$ is an off-diagonal degenerate collision.
\item There exist $x\in K$, $y\in Q_0$, and $v\in S_K(x)$
such that $(x,y,v)$ is an infinitesimal degenerate collision.
\end{enumerate}
\end{theorem}

\begin{proof}
Suppose that Peres-Schlag transversality of order zero fails.
By Proposition~\ref{prop:PS-quantitative}, there exist sequences
$x_j,x'_j\in K$, with $x_j\neq x'_j$, and $y_j\in Q_0$ such that
\[
|\Psi_{y_j}(x_j,x'_j)|
+\sigma_{\min}(D_y\Psi_{y_j}(x_j,x'_j))
\longrightarrow 0.
\]
Since $K$ and $Q_0$ are compact, after passing to subsequences,
we may assume that
\[
x_j\to x,\qquad x'_j\to x',\qquad y_j\to y,
\]
for some $x,x'\in K$ and $y\in Q_0$.

We consider two cases. First, if $x\neq x'$, then by continuity,
\[
\Psi_y(x,x')=0,\qquad
\sigma_{\min}(D_y\Psi_y(x,x'))=0.
\]
Consequently,
\[
\Phi(x,y)=\Phi(x',y),\qquad
\rank\Cphi(x,x';y)<p.
\]
Thus $(x,x',y)$ is an off-diagonal degenerate collision.

Suppose now that $x=x'$. Since the unit sphere is compact,
we may pass to a further subsequence such that
\[
v_j=\frac{x_j-x'_j}{|x_j-x'_j|}
\longrightarrow v\in S_K(x).
\]
For sufficiently large $j$, the segment joining $x_j$ and
$x'_j$ lies in $X$. By the fundamental theorem of calculus, one has
\begin{align*}
\Psi_{y_j}(x_j,x'_j)
&=\int_0^1
D_x\Phi(x'_j+t(x_j-x'_j),y_j)v_j\,dt,\\
D_y\Psi_{y_j}(x_j,x'_j)
&=\int_0^1
D_yD_x\Phi(x'_j+t(x_j-x'_j),y_j)v_j\,dt.
\end{align*}
Since $x_j,x'_j\to x$, $y_j\to y$, and $v_j\to v$,
continuity gives
\[
\Psi_{y_j}(x_j,x'_j)\longrightarrow D_x\Phi(x,y)v
\]
and
\[
D_y\Psi_{y_j}(x_j,x'_j)
\longrightarrow \Tphi(x,y;v).
\]
By our choice of sequences and the continuity of singular
values, it follows that
\[
D_x\Phi(x,y)v=0,\qquad
\rank\Tphi(x,y;v)<p.
\]
Hence $(x,y,v)$ is an infinitesimal degenerate collision, proving the first implication.

Conversely, an off-diagonal degenerate collision implies failure
of transversality by Proposition~\ref{prop:PS-offdiag}.
Suppose instead that $(x,y,v)$ is an infinitesimal degenerate
collision with $v\in S_K(x)$. By \eqref{eq:PS-secant-limit},
choose distinct $x_j,x'_j\in K$ such that
\[
x_j,x'_j\to x,\qquad
\frac{x_j-x'_j}{|x_j-x'_j|}\to v.
\]
The fundamental theorem of calculus gives
\[
\Psi_y(x_j,x'_j)\to D_x\Phi(x,y)v=0,\qquad
D_y\Psi_y(x_j,x'_j)\to\Tphi(x,y;v).
\]
Since $\rank\Tphi(x,y;v)<p$, the corresponding Gram determinants
$\calG_y(x_j,x'_j)$ tend to zero. Thus, for every $C_0>0$ and all
sufficiently large $j$,
\[
2|\Psi_y(x_j,x'_j)|\leq C_0,
\qquad \calG_y(x_j,x'_j)<C_0^2.
\]
Taking $y_1=y_2=y$ contradicts transversality of order zero.
This proves the converse without an interior-point assumption.
\end{proof}

\subsection{Nonlinear projection results beyond Peres--Schlag transversality}

We now give sufficient dimension bounds in terms of the rank of an
augmented mixed Hessian matrix associated with the phase of the
generalized Radon transform. This matrix incorporates both the
incidence constraint $\Phi(x,y)=\bt$ and the mixed second derivatives
of $\Phi$. These dimension bounds apply even when the collision criteria
above show that Peres-Schlag transversality fails.

Let $\Phi:X\times Y\to\R^p$ satisfy \eqref{eq: rank_condition},
as in the preceding sections. For a set $A\subset X$, $y\in Y$, we write
\begin{align*}
     \Phi_y(A)=\{\Phi(x,y):x\in A\}.
\end{align*}
For $\bt\in\R^p$, consider the phase
\[
 \phi_{\bt}(x,y,\tau)
 =\tau\cdot(\Phi(x,y)-\bt),
 \qquad \tau\in\R^p\setminus\bzero.
\]
Its mixed derivative in the variable groups $(\tau,y)$ and
$(\tau,x)$ is a $(p+d_Y)\times(p+d_X)$ matrix
\begin{align*}
     \Bphi(x,y,\tau)
 &:=D_{(\tau,y)}D_{(\tau,x)}\phi_{\bt}(x,y,\tau)=
 \begin{bmatrix}
 0&D_x\Phi(x,y)\\
 D_y\Phi(x,y)^T&D_yD_x(\tau\cdot\Phi)(x,y)
 \end{bmatrix}.
\end{align*}
We call $\Bphi$ the augmented mixed Hessian matrix of the phase.
It is independent of $\bt$.

\begin{theorem}\label{thm:PS-rank-threshold}
Let $A\subset X$ be compact, and write $ s=\dim_HA>0$. Set $d_{\min}=\min\{d_X,d_Y\}$.
Suppose that, for some integer $ 2p\leq  r\leq  p+d_{\min}$, one has
\begin{equation*}
 \rank(\Bphi(x,y,\tau))\geq  r
\end{equation*}
for every $(x,y)\in X\times Y$ and every
$\tau\in\R^p\setminus\bzero$. Then the following hold.

\begin{enumerate}[label=\textup{(\roman*)}]
\item If $0<s<p$, then, for every $0<u<s$,
\begin{equation}\label{eq:PS-rank-dimension}
 \dim_H\bigl\{
 y\in Y:\dim_H\Phi_y(A)\leq  u
 \bigr\}
 \leq  d_X+d_Y+p-r-s+u.
\end{equation}

\item If $s>p$, then
\begin{equation*}
 \dim_H\bigl\{
 y\in Y:\calL^p(\Phi_y(A))=0
 \bigr\}
 \leq  d_X+d_Y+2p-r-s.
\end{equation*}

\item If $s>2p$, then
\begin{equation*}
 \dim_H\bigl\{
 y\in Y:\Int(\Phi_y(A))=\emptyset
 \bigr\}
 \leq  d_X+d_Y+3p-r-s.
\end{equation*}

\item Let $B\subset Y$ be compact and put
\[
 \kappa_r
 =\dim_HA+\dim_HB-d_X-d_Y-2p+r.
\]
If $\kappa_r>0$, then, for every
$0<\e<\kappa_r$, there exists a Frostman probability
measure $\lambda_\e$ supported on $B$ such that, for
$\lambda_\e$-a.e.\ $y$,
\[
 \calL^p\left(
 \left\{\bt\in\R^p:
 \dim_H\bigl(
 A\cap\Phi(\cdot,y)^{-1}(\bt)
 \bigr)\geq \kappa_r-\e
 \right\}\right)>0.
\]
\end{enumerate}
\end{theorem}

\begin{proof}
We first obtain the Sobolev smoothing order $\alpha_\Phi$ associated with the rank assumption. Recall that, for $\bt\in\R^p$, the incidence relation
\[
Z_{\bt}=\{(x,y)\in X\times Y:\Phi(x,y)=\bt\}
\]
is a smooth submanifold of codimension $p$. The associated
canonical relation is the twisted conormal bundle
\begin{align*}
\Lambda_\bt=\bigl\{&
(x,D_x\Phi(x,y)^T\tau;\,y,-D_y\Phi(x,y)^T\tau):\Phi(x,y)=\bt,\quad\tau\in\R^p\setminus\{0\}\bigr\}.
\end{align*}
This is a subset of
$(T^*X\setminus\{0\})\times(T^*Y\setminus\{0\})$.
Let $\pi_L$ and $\pi_R$ be the
cotangent projections from the canonical relation $\Lambda_\bt$
of $\calR_\bt$ onto $T^*X$ and $T^*Y$, respectively. 

Fix a point of $\Lambda_{\bt}$ parametrized by
$(x,y,\tau)$. A tangent vector
$(\dot x,\dot y,\dot\tau)$ satisfies 
\[D_x\Phi\,\dot x+D_y\Phi\,\dot y=0.\]
Since
$\pi_R(x,y,\tau)
=(y,-D_y\Phi(x,y)^T\tau)$,
this tangent vector belongs to $\Ker( D\pi_R)$ if
and only if
\begin{align*}
    \begin{cases}
        \dot y=0,\\
D_yD_x(\tau\cdot\Phi)\dot x+D_y\Phi^T\dot\tau
=0.
    \end{cases}
\end{align*}
The tangent equation then gives $D_x\Phi\,\dot x=0$.
Hence $\Ker( D\pi_R)$ is isomorphic to the kernel of
\[
\Bphi(x,y,\tau)=
\begin{bmatrix}
0&D_x\Phi\\
D_y\Phi^T&D_yD_x(\tau\cdot\Phi)
\end{bmatrix}.
\]
Similarly, $\Ker (D\pi_L)$ is isomorphic to
$\Ker(\Bphi^T)$. Consequently,
\begin{align*}
\dim\Ker( D\pi_R)
&=p+d_X-\rank(\Bphi),\\
\dim\Ker (D\pi_L)
&=p+d_Y-\rank(\Bphi).
\end{align*}

Suppose that $\rank\Bphi(x,y,\tau)\geq  r$ uniformly
on $\Lambda_\bt$, where
$2p\leq  r\leq  p+d_{\min}$ and
$d_{\min}=\min\{d_X,d_Y\}$.
Writing
\[
q_\Phi=p+d_{\min}-\rank\Bphi,
\]
we have $0\leq  q_\Phi\leq  p+d_{\min}-r$.

Recall that the effective order of $\calR_{\bt}$ is $m_{\mathrm{eff}}=\frac{p-d_{\min}}{2}$. We define 
\begin{align*}
    \alpha_r=-m_{\mathrm{eff}}
-\frac{p+d_{\min}-r}{2}
=\frac{r-2p}{2}.
\end{align*}
Since the spatial projections are submersions, 
Theorem~\ref{thm: FIO_standard}(ii) gives
\begin{equation}\label{eq:PS-rank-smoothing}
\calR_{\bt}:L_\sigma^2(X)
\longrightarrow L_{\sigma+\alpha_r}^2(Y),
\qquad \sigma\in\R,
\end{equation}
uniformly for $\bt$ in each fixed compact parameter range, after compact localization of the amplitude.

Let $\mu$ and $\lambda$ be Frostman probability measures supported on
$A$ and on a compact set $C\subset Y$, respectively,
with Frostman exponents $s_\mu$ and $s_\lambda$,
$I_{s_\mu}(\mu)<\infty$ and $I_{s_\lambda}(\lambda)<\infty$. Define $ \nu_y=(\Phi_y)_*\mu$
and put
\[
 \eta
 =\alpha_r
 +\frac{s_\mu-d_X}{2}
 +\frac{s_\lambda-d_Y}{2}.
\]
Let $0<u<p$. Applying Remark~\ref{rem:averaged-sobolev}
with $\beta=(u-p)/2$ gives
\begin{equation}\label{eq:PS-average-image-energy}
 \int_C I_u(\nu_y)\,d\lambda(y)<\infty
\end{equation}
provided $\eta>(u-p)/2$.
The condition on $\eta$ is equivalent to
\begin{equation*}
 s_\mu+s_\lambda >d_X+d_Y+p-r+u.
\end{equation*}

We now prove \textup{(i)}. Assume that $\dim_HA=s\in (0,p)$, and for each $u\in (0,s)$, we define
\[
 E_u=\{y\in Y:\dim_H\Phi_y(A)\leq  u\}.
\]
Suppose, to the contrary, that $\dim_HE_u>d_X+d_Y+p-r-s+u$.
Then choose a compact $C\subset E_u$ with $\dim_HC> d_X+d_Y+p-r-s+u$. By Frostman's lemma, the measures $\mu$ and $\lambda$ can be chosen to be Frostman measures with exponents
$s_\mu<s$, $s_\lambda<\dim_HC$, and a number
$\zeta\in(u,p)$ sufficiently close to $u$ such that
\[
 s_\mu+s_\lambda
 >d_X+d_Y+p-r+\zeta.
\]
It follows from \eqref{eq:PS-average-image-energy} that
$I_\zeta(\nu_y)<\infty$ for $\lambda$-a.e.\ $y$. Hence
\[
 \dim_H\Phi_y(A)
 \geq \dim_H(\supp\nu_y)
 \geq  \zeta>u
\]
for $\lambda$-a.e.\ $y\in C$, contradicting
$C\subset E_u$. This proves \eqref{eq:PS-rank-dimension}.

The proofs of \textup{(ii)} and \textup{(iii)} are similar
to that of \textup{(i)}. Indeed, if either exceptional set
had Hausdorff dimension greater than the asserted bound,
we could choose Frostman measures $\mu$ and $\lambda$ such
that $\eta>0$ in \textup{(ii)}, or $\eta>p/2$ in
\textup{(iii)}. Proposition~\ref{prop: pinned_configuration}
would then imply that $\Phi_y(A)$ has positive Lebesgue
measure or nonempty interior, respectively, for
$\lambda$-almost every $y$ in the corresponding exceptional
set. Both conclusions give a contradiction.

Finally, since
\[
\kappa_r
=\dim_H A+\dim_H B-d_X-d_Y+2\alpha_r>0,
\]
part \textup{(iv)} follows from
Theorem~\ref{thm: pinned_slicing} and
\eqref{eq:PS-rank-smoothing}. This finishes the proof of the theorem.
\end{proof}
The following example shows that Theorem~\ref{thm:PS-rank-threshold}
can give nontrivial pinned configuration estimates even when
Peres--Schlag transversality fails locally and the associated
canonical relation is degenerate. In this example, the rank
of $\Bphi$ is constant and strictly less than its maximal
possible value.

\begin{example}
Let $d\geq 4$, $2\leq  k\leq  d-2$, and
\[
\Phi(x,y)=x_d-y_d-\sum_{j=1}^k x_jy_j,
\qquad x,y\in\R^d.
\]
Both partial differentials have rank one, since
$\partial_{x_d}\Phi=1$ and $\partial_{y_d}\Phi=-1$.
If $x'=x+ae_{k+1}$, $a\ne0$, then
\[
\Phi(x,y)=\Phi(x',y),
\qquad
D_y\Phi(x,y)=D_y\Phi(x',y)
\]
for every $y$. Hence Proposition~\ref{prop:PS-offdiag}
shows that Peres--Schlag transversality fails on every
input set containing such a pair, for every
$0\leq \beta<1$. In particular, such pairs occur in
every nonempty open input set.

To compute the rank of $\Bphi$, observe that
\[
D_x\Phi=(-y_1,\ldots,-y_k,0,\ldots,0,1),\quad
D_y\Phi=(-x_1,\ldots,-x_k,0,\ldots,0,-1),
\]
and
\[
D_yD_x(\tau\cdot\Phi)
=-\tau\operatorname{diag}(I_k,0_{d-k}).
\]
Since $\tau\ne0$, the equation
\[
\Bphi(x,y,\tau)\binom{a}{v}=0,
\qquad a\in\R,\quad v\in\R^d,
\]
implies $a=0,v_1=\cdots=v_k=v_d=0$.
The remaining $d-1-k$ components of $v$ are arbitrary.
Hence
\[
\rank(\Bphi)=k+2,
\qquad
q_\Phi=d-1-k>0.
\]

By Theorem~\ref{thm:PS-rank-threshold}, whenever
\[
\dim_H A+\dim_H B>2d-k,
\]
there exists
$y\in B$ such that $\calL^1(\Phi_y(A))>0$.
Moreover, whenever
\[
\dim_H A+\dim_H B>2d+1-k,
\]
the image set $\Phi_y(A)$ has nonempty interior for some
$y\in B$.
Since $k\geq 2$, both thresholds are strictly less than $2d$.

\end{example}

\section{Dimensions of slices}\label{sec: slicing_sets}
In this section we prove Theorem \ref{thm: pinned_slicing}. The argument follows the
slicing measure method of Iosevich and Liu \cite{Iosevich_Liu_pinned_distance_2019},
adapted to the FIO estimate \eqref{eq: Rt_smoothing_condition_intro}.

Assume that $A\subset X$ and $B\subset Y$ are compact sets such that 
\begin{align*}
    \dim_HA+\dim_HB>d_X+d_Y-2\alpha_\Phi.
\end{align*}
Let $\mu$ and $\lambda$ be Frostman probability measures supported on $A$ and $B$, with exponents $s_\mu$ and $s_\lambda$, respectively, with
\[ 0<s_\mu<\dim_HA, \qquad 0<s_\lambda<\dim_HB.
\]
We further assume that $I_{s_\mu}(\mu)<\infty$ and $I_{s_\lambda}(\lambda)<\infty$.

Recall from Proposition \ref{prop: pinned_configuration} that we define
\begin{align}\label{eq: eta}
    \eta=\alpha_\Phi+\frac{s_\mu-d_X}{2}+\frac{s_\lambda-d_Y}{2}.
\end{align}
We assume throughout this section that $\eta>0$.

By Proposition \ref{prop: pinned_configuration}, for $\lambda$-a.e. $y$, the pinned measure $\nu_y=(\Phi(\cdot,y))_\#\mu$ has an $L^2$ density $d\nu_y(\bt)=h_y(\bt)\,d\bt$.
The disintegration theorem \cite[Theorem 5.3.1]{Ambrosio_2008gradient} gives conditional probability measures $P_{y,\bt}$ supported on $A\cap\{x:\Phi(x,y)=\bt\}$ for $\nu_y$-a.e. $\bt$.
Define the weighted slicing measures by
\begin{align*}
\mu_{y,\bt}=h_y(\bt)P_{y,\bt},
\end{align*}
and set $\mu_{y,\bt}=0$ when $h_y(\bt)=0$. Then
\begin{align}\label{eq: slice_measure_support}
\supp\mu_{y,\bt} \subset A\cap\{x:\Phi(x,y)=\bt\}, \qquad \mu_{y,\bt}(A)=h_y(\bt),
\end{align}
for $(\lambda\otimes\calL^p)$-a.e. $(y,\bt)$. Moreover, for $\lambda$-a.e. $y$ and bounded Borel functions $f$ and $\varphi$,
\begin{align}\label{eq: slice_disintegration}
\int_{\R^p}\varphi(\bt) \left(\int_A f(x)\,d\mu_{y,\bt}(x)\right)d\bt= \int_A f(x)\varphi(\Phi(x,y))\,d\mu(x).
\end{align}
We choose the densities and conditional probabilities jointly measurably by disintegrating $\lambda(dy)\mu(dx)$ under $(y,x)\mapsto(y,\Phi(x,y))$; its joint pushforward has density $h_y(\bt)$ with respect to $\lambda(dy)\,d\bt$. This choice allows us to use Tonelli later.
This construction is analogous to slicing measures by affine planes; see \cite[Chapter 10]{Mattila_1995}. For slicing measures associated with nonlinear level sets, see \cite[Section 3]{Iosevich_Liu_pinned_distance_2019}.

 We have the following averaged energy estimate.
\begin{lemma}\label{lem: slice_measure_energy}
Suppose $\eta>0$. Then, for every $0<\sigma<2\eta$, one has
\begin{align}\label{eq: slice_measure_energy}
\int_B\int_{\R^p} I_\sigma(\mu_{y,\bt})\,d\bt\,d\lambda(y)<\infty.
\end{align}
\end{lemma}

\begin{proof}
Fix $0<\sigma<2\eta$. Our goal is to show that 
\begin{align*}
    \mathcal{J}=\int_B\int_{\R^p} I_\sigma(\mu_{y,\bt})\,d\bt\,d\lambda(y)<\infty.
\end{align*}
To show this, write
\begin{align*}
   \mathcal{J}= \int_B\int_{\R^p} I_\sigma(\mu_{y,\bt})\,d\bt\,d\lambda(y)=c(d_X,\sigma)\int_B\int_{\R^p} \left(\int_{\R^{d_X}}\vert \widehat{\mu_{y,\bt}}(\xi)|^2|\xi|^{\sigma-d_X}\,d\xi\right)\,d\bt\,d\lambda(y).
\end{align*}
Applying Tonelli's theorem, we can write
\begin{align}\label{eq: energy_xi}
    \mathcal{J}=c(d_X,\sigma)\int_{\R^{d_X}}|\xi|^{\sigma-d_X}\left[\int_B\int_{\R^p} \vert \widehat{\mu_{y,\bt}}(\xi)|^2\,d\bt\,d\lambda(y)\right]\,d\xi.
\end{align}
For $\xi\in\R^{d_X}$, define the complex measure $d\mu_\xi(x) =e^{-2\pi i x\cdot\xi}\,d\mu(x)$.
From \eqref{eq: slice_disintegration}, one has, in the sense of distributions,
\begin{align*}
\widehat{\mu_{y,\bt}}(\xi) = \calR_\bt\mu_\xi(y).
\end{align*}
We next give an estimate for \[\int_{\R^p}\Vert\calR_\bt\mu_\xi\Vert_{L^2(\lambda)}^2\,d\bt.\]
Choose $\delta>0$ sufficiently small so that $\sigma+2\delta<2\eta$, and set
\begin{align}\label{eq: parameters_r_a_0}
    r=\frac{d_Y-s_\lambda}{2}+\delta,\qquad a_0=\alpha_\Phi-r=\alpha_\Phi-\frac{d_Y-s_\lambda}{2}-\delta.
\end{align}
By \eqref{eq: eta} and the choice of $\delta$, one has \[2a_0=d_X-s_\mu+2\eta-2\delta>d_X-s_\mu.\]
By Lemma \ref{lem: Frostman_LP} (ii) and the Littlewood--Paley decomposition, for a smooth function $h$,
\begin{align} \label{eq:trace-lambda}
\Vert h\Vert_{L^2(d\lambda)} \lesssim_\delta \Vert h\Vert_{L^2_r(Y)}.
\end{align}
Combining \eqref{eq:trace-lambda} with the Sobolev estimate for
$\calR_\bt$ \eqref{eq: Rt_smoothing_condition_intro} one has, uniformly for $\bt\in W$,
\begin{align}\label{eq: Rt_estimate}
\Vert\calR_\bt f\Vert_{L^2(d\lambda)} \lesssim \Vert \calR_\bt f\Vert_{L^2_r(Y)}
\lesssim \Vert f\Vert_{L^2_{r-\alpha_\Phi}(X)} = \Vert f\Vert_{L^2_{-a_0}(X)}.
\end{align}
We extend $\calR_\bt$ to a bounded operator $\calR_\bt: L_{-a_0}^2(X)\to L^2(d\lambda)$.

Observe that, for every fixed $\xi$, $\widehat{\mu_\xi}(\zeta)=\widehat{\mu}(\xi+\zeta)$. Using the fact that $\langle\omega\rangle\lesssim \langle\omega-\xi\rangle\langle \xi\rangle$, and  $2a_0>d_X-s_\mu$, one has
\begin{align*}
    \Vert \mu_\xi\Vert_{L^2_{-a_0}}^2
    &=\int \vert \widehat\mu (\zeta+\xi)|^2\langle \zeta\rangle^{-2a_0}\,d\zeta\\
    &=\int |\widehat{\mu}(\omega)|^2\langle \omega-\xi\rangle^{-2a_0}\,d\omega\\
    &\lesssim \langle \xi\rangle^{2a_0}\int \vert \widehat{\mu}(\omega)|^2\langle\omega\rangle^{-2a_0}\,d\omega\\
    &\lesssim \langle \xi\rangle^{2a_0}I_{s_\mu}(\mu)<\infty.
\end{align*}
Thus, for every fixed $\xi$, $\mu_\xi\in L^2_{-a_0}(X)$, and hence, from \eqref{eq: Rt_estimate}, one has
\begin{align*}
\Vert\calR_\bt\mu_\xi\Vert_{L^2(d\lambda)} \lesssim
\Vert\mu_\xi\Vert_{L^2_{-a_0}(X)},
\end{align*}
uniformly for $\bt\in W$. As a function of $\bt$, $\calR_\bt\mu_\xi$ vanishes outside the fixed bounded set $\Phi(\supp\chi)\subset W$, squaring and integrating in $\bt$ yields
\begin{align*}
\int_{\R^p} \Vert \calR_\bt\mu_\xi\Vert _{L^2(d\lambda)}^2\,d\bt \lesssim \Vert \mu_\xi\Vert _{L^2_{-a_0}(X)}^2.
\end{align*}
Therefore, for $\xi\in \R^{d_X}$, 
\begin{align*}
    \int_B\int_{\R^p} \vert \widehat{\mu_{y,\bt}}(\xi)|^2\,d\bt\,d\lambda(y) \lesssim \Vert\mu_\xi\Vert_{L^2_{-a_0}(X)}^2.
\end{align*}
Substituting this bound into \eqref{eq: energy_xi}, we obtain
\begin{align*}
\mathcal{J}\lesssim \int_{\R^{d_X}} |\xi|^{\sigma-d_X} \Vert\mu_\xi\Vert_{L^2_{-a_0}(X)}^2\,d\xi=\int_{\R^{d_X}}|\xi|^{\sigma-d_X}\int_{\R^{d_X}}\vert\widehat{\mu_\xi}(\zeta)|^2\langle \zeta\rangle^{-2a_0}\,d\zeta \,d\xi.
\end{align*}
Since $\widehat{\mu_\xi}(\zeta)= \widehat\mu(\zeta+\xi)$, by Tonelli and the change of variables $\omega=\zeta+\xi$, the expression on the right equals
\begin{align}\label{eq: energy_mu_omega}
\int_{\R^{d_X}} |\widehat\mu(\omega)|^2
\left( \int_{\R^{d_X}} |\xi|^{\sigma-d_X} \langle\omega-\xi\rangle^{-2a_0}\,d\xi \right)d\omega.
\end{align}
From the choice of $a_0$ in
\eqref{eq: parameters_r_a_0}, we have
\[
0<\sigma<2a_0<d_X.
\]
Indeed, $\sigma+2\delta<2\eta$ implies $\sigma<2a_0$, while
\[
2a_0=2\alpha_\Phi-(d_Y-s_\lambda)-2\delta
<2\alpha_\Phi\leq\min\{d_X,d_Y\}-p<d_X.
\]
These inequalities check both the local integrability and the integrability
at infinity required for the convolution below.

By the classical Riesz-kernel convolution formula,
which follows from
\cite[Theorem~3.6]{Mattila2015},
\[
\int_{\R^{d_X}}
|\xi|^{\sigma-d_X}
|\omega-\xi|^{-2a_0}\,d\xi
=C|\omega|^{\sigma-2a_0},
\qquad \omega\neq0.
\]
Since
$\langle\omega-\xi\rangle^{-2a_0}
\leq|\omega-\xi|^{-2a_0}$,
it follows from \eqref{eq: energy_mu_omega} that
\[
\mathcal J
\lesssim
\int_{\R^{d_X}}
|\widehat\mu(\omega)|^2
|\omega|^{\sigma-2a_0}\,d\omega.
\]
Moreover, since $\sigma+2\delta<2\eta$, one has
\[
\sigma-2a_0
=s_\mu-d_X+(\sigma-2\eta+2\delta)
<s_\mu-d_X.
\]
The integral over $|\omega|\geq1$ is therefore
bounded by a constant multiple of
$I_{s_\mu}(\mu)$. The integral over
$|\omega|<1$ is finite since
$|\widehat\mu|\leq1$ and
$\sigma-2a_0>-d_X$.
Consequently,
\[
\mathcal J\lesssim
1+I_{s_\mu}(\mu)<\infty.
\]
This proves \eqref{eq: slice_measure_energy} and completes the proof of the lemma.
\end{proof}

\begin{proof}[Proof of Theorem \ref{thm: pinned_slicing}]
Let $A\subset X$, and $B\subset Y$ be compact sets of positive Hausdorff dimensions. Set $ s_A=\dim_HA, s_B=\dim_HB$, and assume that 
\[ \kappa = s_A+s_B-d_X-d_Y+2\alpha_\Phi>0. \]
Fix $0<\e<\kappa$ and choose
\[ 0<s_\mu<s_A, \qquad 0<s_\lambda<s_B
\]
sufficiently close to $s_A,s_B$ that
\begin{align*}
2\eta=s_\mu+s_\lambda-d_X-d_Y+2\alpha_\Phi>\kappa-\e>0.
\end{align*}
Choose Frostman probability measures $\mu$ on $A$ and $\lambda_\e$ on $B$ with Frostman exponents $s_\mu$ and $s_\lambda$, respectively, such that
\[ I_{s_\mu}(\mu)<\infty,\qquad I_{s_\lambda}(\lambda_\e)<\infty. \]
Proposition \ref{prop: pinned_configuration} gives, for $\lambda_\e$-a.e.  $y$, the pushforward measure $\nu_y=(\Phi_{y})_\#\mu$ has an $L^2$ density $h_y$. In particular,
\[ \int_{\R^p}h_y(\bt)\,d\bt=1,\]
and the set $E_y:=\{\bt\in\Delta_\Phi^y(A):h_y(\bt)>0\}$ has positive Lebesgue measure.

Choose $\kappa-\e<\sigma<2\eta$. Then by Lemma \ref{lem: slice_measure_energy} and Fubini, for $\lambda_\e$-a.e. $y$ and $\calL^p$-a.e.
$\bt$,
\[ I_\sigma(\mu_{y,\bt})<\infty. \]
For a.e. $\bt \in E_y$, the slicing measure $\mu_{y,\bt}$ is nonzero by \eqref{eq: slice_measure_support}, and hence the energy criterion gives
\[ \dim_H(\supp\mu_{y,\bt})\geq \sigma. \]
Since $\supp\mu_{y,\bt}\subset A\cap\{x:\Phi(x,y)=\bt\}$, we obtain
\[\dim_H\{x\in A:\Phi(x,y)=\bt\}\geq \sigma>\kappa-\e\] for a.e.  $\bt\in E_y$.

Therefore
\[\calL^p\left(\left\{\bt\in\R^p:\dim_H\{x\in A:\Phi(x,y)=\bt\}\geq \kappa-\e\right\}\right)>0\]
for $\lambda_\e$-a.e. $y\in B$. Since $E_y\subset\Delta_\Phi^y(A)$, one also has
\[
\calL^p(\Delta_\Phi^y(A))>0.
\]
This completes the proof of the theorem.
\end{proof}

\section{Restricted projections without Peres-Schlag transversality}
\label{sec: projection_nontranversal_interior}

We now apply Theorem~\ref{thm: main_pinned_configuration}
to the restricted family of projections introduced by Harris.
First, we observe that the parameter derivatives $D_v\pi_v$ need not have rank two,
so the family does not directly satisfy the submersion
conditions in \eqref{eq: rank_condition}.
To overcome this, we introduce an auxiliary translation
parameter $r\in\R^2$ and define
\begin{equation*}
\Phi_H(x,(v,r))=\pi_v(x)-r,
\qquad
x\in\R^{d+1},\quad
(v,r)\in S^{d-1}_+\times\R^2.
\end{equation*}
Since $D_r\Phi_H=-I_2$, we have $\rank D_{(v,r)}\Phi_H=2$
everywhere. Moreover, for every $E\subset\R^{d+1}$,
\[
\Delta_{\Phi_H}^{(v,r)}(E)=\pi_v(E)-r.
\]
Thus the additional parameter does not affect whether
the projection has positive Lebesgue measure or nonempty
interior.

Using the chart $v(q)=(q,\sqrt{1-|q|^2})$, $q\in B^{d-1}(0,1)$, we identify $Y$ with the open subset $B^{d-1}(0,1)\times\R^2$ of $\R^{d+1}$.
This identification preserves Hausdorff dimension on compact subsets. We therefore take $X=\R^{d+1}$, and $Y=S^{d-1}_+\times\R^2$ for convenience.
Since $\dim S^{d-1}_+=d-1$, we have
\begin{align*}
d_X=d_Y=d+1,
\qquad
p=2.
\end{align*}
Thus $\Phi_H: X\times Y\to \R^2$ is a smooth map satisfying
\begin{align*}
    \rank D_x\Phi_H=2, \quad \rank D_{v,r}\Phi_H=2,
\end{align*}
everywhere.
The remaining input we need to apply Theorem \ref{thm: main_pinned_configuration} is the Sobolev regularity of the associated generalized Radon transforms
\begin{align*}
    \calR_\bt^Hf(y)=\int\chi(x,y)f(x)\,d\sigma_{y,\bt}(x),
\end{align*}
where $\sigma_{y,\bt}$ is the Leray measure on $\{x: \Phi_H(x,y)=\bt\}$, $\bt\in \R^2$. Here $\chi$ is a smooth cutoff function with compact support. Let $W\subset \R^2$ be a bounded open set such that $\Phi_H(\supp\chi)\subset\subset W$.

For fixed $\bt\in\R^2$, the incidence relation is $Z_{\bt}=\{(x,v,r):\Phi_H(x,(v,r))=\bt\}$. Consider the phase
\begin{align*}
\phi_{\bt}(x,v,r,\tau) = \tau\cdot\bigl(\Phi_H(x,(v,r))-\bt\bigr),
\qquad \tau\in\R^2\setminus\bzero,
\end{align*}
the associated canonical relation is
\begin{align}\label{eq: Harris_canonical_relation}
\Lambda_\bt^H:=\left\{\left(x,d_x\phi_{\bt};(v,r),-d_{(v,r)}\phi_{\bt}\right): d_\tau\phi_\bt=0\right\}\subset (T^\ast X\setminus \bzero)\times (T^\ast Y\setminus \bzero).
\end{align}

\begin{lemma}
\label{lem:Harris-smoothing}
The canonical relation $\Lambda_{\bt}^{H}$ associated with $\calR_\bt^H$ has a one-sided Whitney fold.
Consequently, the associated generalized Radon transforms satisfy
\begin{align*}
\mathcal R_{\bt}^{H}: L^2_s(\R^{d+1}) \longrightarrow L^2_{s+\alpha_{\Phi_H}}(S^{d-1}_+\times\R^2), \qquad s\in\R,
\end{align*}
uniformly for $\bt\in W$, where
\begin{align*}
\alpha_{\Phi_H} = \frac{d-1}{2}-\frac14 = \frac{2d-3}{4}.
\end{align*}
\end{lemma}

    \begin{proof}
In the parametrization of \eqref{eq: Harris_canonical_relation},
the left cotangent projection is
\[
\pi_L(x,v,r,\tau)
=\left(x,\left(\frac{\tau_1}{\sqrt2}v+\tau_2Av,
-\frac{\tau_1}{\sqrt2}\right)\right),
\qquad\Phi_H(x,(v,r))=\bt.
\]
Since the $x$-component is the identity, the singularities of
$\pi_L$ are determined by the map
\begin{align*}
F(v,\tau_1,\tau_2) = \left( \frac{\tau_1}{\sqrt2}v+\tau_2Av, -\frac{\tau_1}{\sqrt2} \right).
\end{align*}
Fix $v,\tau_1,\tau_2$, we describe $\Ker(DF)(v,\tau_1,\tau_2)$. Let $u\in T_vS^{d-1}$, and $(\dot\tau_1,\dot\tau_2)\in\R^2$, differentiating gives
\begin{align*}
DF(u,\dot\tau_1,\dot\tau_2) =
\left( \frac{\tau_1}{\sqrt2}u +\tau_2Au +\frac{\dot\tau_1}{\sqrt2}v +\dot\tau_2Av, -\frac{\dot\tau_1}{\sqrt2} \right).
\end{align*}
If $(u,\dot\tau_1,\dot\tau_2)\in \Ker(DF)$, then
\begin{align}\label{eq: DF=0}
    \begin{cases}
    \dot\tau_1&=0\\
    \left(\frac{1}{\sqrt{2}}\tau_1I +\tau_2A\right)u+\dot\tau_2Av&=0
\end{cases}.
\end{align}
Now observe that by the assumptions on $A$, we have $A^{-1}=A^*=-A$ and $A^2=-I$, combining with the fact that $(\tau_1,\tau_2)\neq (0,0)$, we see that $\left(\frac{1}{\sqrt{2}}\tau_1I +\tau_2A\right)$ is invertible. Indeed, using $A^2=-I$, one can check that
\begin{align*}
    \left(\frac{\tau_1}{\sqrt{2}}I +\tau_2A\right)\left(\frac{\tau_1}{\sqrt{2}}I -\tau_2A\right)=\left(\frac{\tau_1^2}{2}+\tau_2^2\right)I.
\end{align*}
Hence from \eqref{eq: DF=0}, one has
\begin{align*}
    u=-\frac{\dot\tau_2}{\frac{\tau_1^2}{2}+\tau_2^2}\left(\frac{\tau_1}{\sqrt{2}}I -\tau_2A\right)Av=-\frac{\dot\tau_2}{\frac{\tau_1^2}{2}+\tau_2^2}\left(\frac{\tau_1}{\sqrt{2}}Av +\tau_2v\right),
\end{align*}
where we used $A^2v=-v$. Recall that $u\in T_vS^{d-1}$, hence $u\cdot v=0$. Using $Av\cdot v=0$, one gets
\begin{align*}
    0=u\cdot v=-\frac{\dot\tau_2}{\frac{\tau_1^2}{2}+\tau_2^2}\left(\frac{\tau_1}{\sqrt{2}}Av +\tau_2v\right)\cdot v=-\frac{\dot\tau_2\tau_2}{\frac{\tau_1^2}{2}+\tau_2^2}.
\end{align*}
Therefore $\tau_2\dot\tau_2=0$. 
If $\tau_2\neq 0$, then $\dot\tau_2=0$, and hence $u=0$. This implies that $\Ker DF=\{\bzero\}$. If $\tau_2=0$, then $\tau_1\neq 0$, and 
\begin{align*}
    u=-\frac{\sqrt{2}\dot\tau_2}{\tau_1}Av.
\end{align*}
Therefore,
\begin{align*}
    \Ker DF=\operatorname{span}\left\{\left(Av,0,-\frac{\tau_1}{\sqrt{2}}\right)\right\}, \quad \tau_2=0.
\end{align*}
Thus the rank of $DF$, and hence of $D\pi_L$, drops by exactly one on the smooth hypersurface
\begin{align*}
\Sigma=\{\tau_2=0\}.
\end{align*}
It remains to verify the fold transversality condition. One computes
\begin{align*}
d\tau_2\left((Av,0,-\tau_1/\sqrt2)\right)=-\frac{\tau_1}{\sqrt2}\neq 0.
\end{align*}
Hence $\Ker (D\pi_L) \pitchfork T\Sigma$, i.e. the null direction is transverse to the singular hypersurface.

To see that the rank drops simply along $\Sigma$, choose an orthonormal basis $e_1,\ldots,e_{d-1}$ of $T_vS^{d-1}$, oriented so that $\det[e_1,\ldots,e_{d-1},v]=1$. Let
\begin{align*}
M=\frac{\tau_1}{\sqrt2}I+\tau_2A.
\end{align*}
Then
\begin{align*}
DF= \begin{bmatrix}
Me_1&\cdots&Me_{d-1}&v/\sqrt2&Av\\
0&\cdots&0&-1/\sqrt2&0
\end{bmatrix}.
\end{align*}
Expanding along the last row and factoring out $M$, we obtain
\begin{align*}
\det DF
&=\frac1{\sqrt2}\det[Me_1,\ldots,Me_{d-1},Av]\\
&=\frac{\det M}{\sqrt2} \det[e_1,\ldots,e_{d-1},M^{-1}Av].
\end{align*}
Since $A^2=-I$, one has
\begin{align*}
M^{-1}Av = \frac{\frac{\tau_1}{\sqrt2}Av+\tau_2v}
     {\frac{\tau_1^2}{2}+\tau_2^2}.
\end{align*}
Because $Av\in T_vS^{d-1}$, the vector $Av$ lies in the span of $e_1,\ldots,e_{d-1}$, and hence
\begin{align*}
\det[e_1,\ldots,e_{d-1},M^{-1}Av] = \frac{\tau_2}{\frac{\tau_1^2}{2}+\tau_2^2}.
\end{align*}
Moreover, in an orthonormal basis adapted to $A$, the matrix $M$ consists of $d/2$ blocks of the form
\begin{align*}
\begin{bmatrix}
\tau_1/\sqrt2&-\tau_2\\
\tau_2&\tau_1/\sqrt2
\end{bmatrix}.
\end{align*}
Thus
\begin{align*}
\det M= \left(\frac{\tau_1^2}{2}+\tau_2^2\right)^{d/2},
\end{align*}
and consequently
\begin{align*}
\det DF = \frac{\tau_2}{\sqrt2} \left(\frac{\tau_1^2}{2}+\tau_2^2\right)^{(d-2)/2}.
\end{align*}
In arbitrary local coordinates on the sphere, this expression is multiplied by a smooth nonvanishing factor. Since $\tau_1\neq0$ on $\Sigma=\{\tau_2=0\}$, the determinant vanishes to first order on this hypersurface. Thus the rank drops simply by one along $\Sigma$.
Thus $\pi_L$ has Whitney fold singularities along $\Sigma$
and is a local diffeomorphism away from $\Sigma$.
Consequently, $\Lambda_\bt^H$ is a one-sided fold
canonical relation.

We now compute the corresponding Sobolev gain $\alpha_{\Phi_H}$. 
Since the amplitude defining $\calR_{\bt}^{H}$ has order zero, its
FIO order is
\begin{align*}
m = \frac{p}{2}-\frac{d_X+d_Y}{4} = 1-\frac{2(d+1)}{4} = -\frac{d-1}{2}.
\end{align*}
Since $d_X=d_Y$, the effective order is also $m_{\mathrm{eff}}=m=-\frac{d-1}{2}$.
Since one of the canonical projections is a Whitney fold, the one-sided fold theorem of Greenleaf--Seeger, Theorem~\ref{thm:one-sided-fold}, applied to $\calR_{\bt}^{H}$ gives a loss of $1/4$ derivative. Therefore
\begin{align*}
\alpha_{\Phi_H} = -m-\frac14 = \frac{d-1}{2}-\frac14 = \frac{2d-3}{4}.
\end{align*}
Hence
\begin{align*}
\calR_{\bt}^{H}: L^2_s(\R^{d+1}) \longrightarrow L^2_{s+\alpha_{\Phi_H}} (S^{d-1}_+\times\R^2),
\qquad \alpha_{\Phi_H} = \frac{2d-3}{4},
\end{align*}
uniformly for $\bt\in W$.
\end{proof}
We now apply Theorem \ref{thm: main_pinned_configuration} to prove Theorem \ref{thm: Harris_projection_interior}.

\begin{proof}[Proof of Theorem \ref{thm: Harris_projection_interior}]
Suppose, to the contrary, that
\begin{align*}
\dim_H\left\{ v\in S^{d-1}_+: \Int(\pi_{v}(E))=\emptyset \right\} > d+\frac72-s.
\end{align*}
We may choose a compact set $\Gamma\subset\left\{ v\in S^{d-1}_+: \Int(\pi_{v}(E))=\emptyset
\right\}$ such that $\dim_H\Gamma>d+\frac{7}{2}-s$. Let $K=[0,1]^2$, and set $B=\Gamma\times K\subset S^{d-1}_+\times\R^2$.

Recall that $d_X=d_Y=d+1$, $p=2$, and $\alpha_{\Phi_H}=\frac{2d-3}{4}$ by Lemma~\ref{lem:Harris-smoothing}, hence we compute
\begin{align*}
    d_X+d_Y-2\alpha_{\Phi_H}+p=2(d+1)-d+\frac{3}{2}+2=d+\frac{11}{2}.
\end{align*}
On the other hand, we have
\begin{align*}
\dim_HE+\dim_HB>s+d+\frac{7}{2}-s+2=d+\frac{11}{2}.
\end{align*}
Therefore, $\dim_HE+\dim_HB>d_X+d_Y-2\alpha_{\Phi_H}+p$, and Theorem \ref{thm: main_pinned_configuration} yields that for some $(v,r)\in B$, 
\begin{align*}
    \Int(\Delta_{\Phi_H}^{(v,r)}(E))\neq \emptyset.
\end{align*}
This implies $\Int(\pi_v(E))\neq \emptyset$, contradicting $v\in\left\{
v\in S^{d-1}_+:
\Int(\pi_{v}(E))=\emptyset
\right\}$.

Therefore
\begin{align*}
\dim_H\left\{ v\in S^{d-1}_+: \Int(\pi_{v}(E))=\emptyset \right\}
\le d+\frac72-s.
\end{align*}
If $s>9/2$, then $d+7/2-s<d-1$. Thus the exceptional
set has zero $\mathcal H^{d-1}$-measure on the hemisphere, as claimed.
This finishes the proof of the theorem.
\end{proof}

\begin{remark}
Harris \cite{Harris2022} proved that
\begin{align*}
\mathcal L^2(\pi_{v}(E))>0
\qquad\text{for a.e. }v\in S^{d-1}_+
\end{align*}
under the sharp hypothesis $\dim_HE>2$.
By contrast, Theorem \ref{thm: main_pinned_configuration} would yield positive measure only for
$\dim_HE>5/2$.

Thus the purpose of the present application is not to recover or
improve Harris's positive measure result. Rather, Theorem \ref{thm: Harris_projection_interior} illustrates that the pinned FIO approach 
yields the stronger qualitative conclusion of nonempty interior, together with the quantitative exceptional set 
estimate \eqref{eq: Harris_projection_interior_exceptional}, for a projection family
outside of the Peres-Schlag transversality framework.
\end{remark}

\end{document}